\documentclass[11pt,letterpaper]{amsart}
\usepackage[usenames,dvipsnames]{color}
\usepackage{amsthm,amsfonts,amssymb,amsmath,amsxtra}
\usepackage[all]{xy}
\SelectTips{cm}{}
\usepackage{xr-hyper}
\usepackage[pdfpagelabels,pdftex,hidelinks]{hyperref}
\hypersetup{
  colorlinks   = true, 
  urlcolor     = blue, 
  linkcolor    = blue, 
  citecolor   = blue, 
pdftitle={},
  pdfauthor={},
  pdfsubject={},
  pdfkeywords={},
  breaklinks=true,
  bookmarksopen=true,
  bookmarksnumbered=true,
  pdfpagemode=UseOutlines,
  plainpages=false}

\usepackage{verbatim}
\usepackage{caption}

\usepackage{lineno}

\usepackage{mathrsfs}

\usepackage{dynkin-diagrams}
\usepackage{booktabs}
\usepackage[figuresright]{rotating}

\RequirePackage{xspace}
\RequirePackage{etoolbox}
\RequirePackage{varwidth}
\RequirePackage{enumitem}
\RequirePackage{tensor}
\RequirePackage{mathtools}
\RequirePackage{longtable}
\RequirePackage{multirow}

\def\ge{\geqslant}
\def\le{\leqslant}
\def\a{\alpha}
\def\b{\beta}
\def\g{\gamma}

\def\d{\delta}
\def\D{\Delta}

\def\o{\omega}

\def\s{\sigma}
\def\t{\tau}

\def\l{\lambda}

\def\i{^{-1}}

\def\<{\langle}
\def\>{\rangle}

\newcommand{\BA}{\ensuremath{\mathbb {A}}\xspace}
\newcommand{\BB}{\ensuremath{\mathbb {B}}\xspace}

\newcommand{\BF}{\ensuremath{\mathbb {F}}\xspace}
\newcommand{{\BG}}{\ensuremath{\mathbb {G}}\xspace}
\newcommand{\BH}{\ensuremath{\mathbb {H}}\xspace}

\newcommand{{\BK}}{\ensuremath{\mathbb {K}}\xspace}

\newcommand{\BM}{\ensuremath{\mathbb {M}}\xspace}

\newcommand{\BQ}{\ensuremath{\mathbb {Q}}\xspace}

\newcommand{\BT}{\ensuremath{\mathbb {T}}\xspace}
\newcommand{\BU}{\ensuremath{\mathbb {U}}\xspace}

\newcommand{\BW}{\ensuremath{\mathbb {W}}\xspace}

\newcommand{\BZ}{\ensuremath{\mathbb {Z}}\xspace}

\newcommand{\CG}{\ensuremath{\mathcal {G}}\xspace}
\newcommand{\CH}{\ensuremath{\mathcal {H}}\xspace}

\newcommand{\CO}{\ensuremath{\mathcal {O}}\xspace}

\newcommand{\CT}{\ensuremath{\mathcal {T}}\xspace}
\newcommand{\CU}{\ensuremath{\mathcal {U}}\xspace}

\newcommand{\Ad}{{\mathrm{Ad}}}

\newcommand{\GL}{\mathrm{GL}}

\DeclareMathOperator{\Hom}{Hom}

\let\Im\relax
\DeclareMathOperator{\Im}{Im}

\DeclareMathOperator{\ord}{ord}

\newcommand{\SL}{{\mathrm{SL}}}
\DeclareMathOperator{\Spec}{Spec}

\def\tx{\tilde x}

\DeclareMathOperator{\supp}{supp}

\def\brk{{\breve k}}
\def\dw{{\dot w}}

\def\pr{{\rm pr}}

\newtheorem{theorem}{Theorem}
\newtheorem{proposition}[theorem]{Proposition}
\newtheorem{lemma}[theorem]{Lemma}

\newtheorem{corollary}[theorem]{Corollary}

\theoremstyle{definition}

\newtheorem{remark}[theorem]{Remark}

\numberwithin{equation}{section}
\numberwithin{theorem}{section}

\setitemize[0]{leftmargin=*,itemsep=\the\smallskipamount}
\setenumerate[0]{leftmargin=*,itemsep=\the\smallskipamount}

\renewcommand{\to}{%
   \ifbool{@display}{\longrightarrow}{\rightarrow}%
   }
\let\shortmapsto\mapsto
\renewcommand{\mapsto}{%
   \ifbool{@display}{\longmapsto}{\shortmapsto}%
   }
\newlength{\olen}
\newlength{\ulen}
\newlength{\xlen}
\newcommand{\xra}[2][]{%
   \ifbool{@display}%
      {\settowidth{\olen}{$\overset{#2}{\longrightarrow}$}%
       \settowidth{\ulen}{$\underset{#1}{\longrightarrow}$}%
       \settowidth{\xlen}{$\xrightarrow[#1]{#2}$}%
       \ifdimgreater{\olen}{\xlen}%
          {\underset{#1}{\overset{#2}{\longrightarrow}}}%
          {\ifdimgreater{\ulen}{\xlen}%
             {\underset{#1}{\overset{#2}{\longrightarrow}}}
             {\xrightarrow[#1]{#2}}}}%
      {\xrightarrow[#1]{#2}}
   }
\makeatother
\newcommand{\xyra}[2][]{%
   \settowidth{\xlen}{$\xrightarrow[#1]{#2}$}%
   \ifbool{@display}%
      {\settowidth{\olen}{$\overset{#2}{\longrightarrow}$}%
       \settowidth{\ulen}{$\underset{#1}{\longrightarrow}$}%
       \ifdimgreater{\olen}{\xlen}%
          {\mathrel{\xymatrix@M=.12ex@C=3.2ex{\ar[r]^-{#2}_-{#1} &}}}%
          {\ifdimgreater{\ulen}{\xlen}%
             {\mathrel{\xymatrix@M=.12ex@C=3.2ex{\ar[r]^-{#2}_-{#1} &}}}
             {\mathrel{\xymatrix@M=.12ex@C=\the\xlen{\ar[r]^-{#2}_-{#1} &}}}}}%
      {\mathrel{\xymatrix@M=.12ex@C=\the\xlen{\ar[r]^-{#2}_-{#1} &}}}%
   }
\makeatletter
\newcommand{\xla}[2][]{%
   \ifbool{@display}%
      {\settowidth{\olen}{$\overset{#2}{\longleftarrow}$}%
       \settowidth{\ulen}{$\underset{#1}{\longleftarrow}$}%
       \settowidth{\xlen}{$\xleftarrow[#1]{#2}$}%
       \ifdimgreater{\olen}{\xlen}%
          {\underset{#1}{\overset{#2}{\longleftarrow}}}%
          {\ifdimgreater{\ulen}{\xlen}%
             {\underset{#1}{\overset{#2}{\longleftarrow}}}
             {\xleftarrow[#1]{#2}}}}%
      {\xleftarrow[#1]{#2}}
   }
\newcommand{\isoarrow}{%
   \ifbool{@display}{\overset{\sim}{\longrightarrow}}{\xrightarrow\sim}%
   }

\begin{document}

\title[Deep level Deligne--Lusztig representations of Coxeter type]{Deep level Deligne--Lusztig representations of Coxeter type}

\author{Alexander B. Ivanov}
\address{Fakult\"at f\"ur Mathematik, Ruhr-Universit\"at Bochum, D-44780 Bochum, Germany.}
\email{a.ivanov@rub.de}

\author{Sian Nie}
\address{Academy of Mathematics and Systems Science, Chinese Academy of Sciences, Beijing 100190, China}

\address{ School of Mathematical Sciences, University of Chinese Academy of Sciences, Chinese Academy of Sciences, Beijing 100049, China}
\email{niesian@amss.ac.cn}

\author{Panjun Tan}
\address{Academy of Mathematics and Systems Science, Chinese Academy of Sciences, Beijing 100190, China}
\email{tanpanjun@amss.ac.cn}

\begin{abstract}In this article we study the cohomology of deep level Deligne--Lusztig varieties of Coxeter type, attached to a reductive group over a local non-archimedean field, which splits over an unramified extension. This allows to construct some new irreducible representations of parahoric subgroups of $p$-adic groups. Moreover, in the quasi-split case we prove that these compactly induce to finite direct sums of irreducible supercuspidal representations of the $p$-adic group. This extends previous results of \cite{DI}, \cite{CI_loopGLn}. \end{abstract}
\maketitle

\section{Introduction}\label{sec:intro}

Let $k$ be a non-Archimedean local field with residue characteristic $p>0$, integers $\CO_k$, uniformizer $\varpi$ and residue field $\BF_q$. Let $\breve k$ be the completion of the maximal unramified extension of $k$, let $\CO_{\brk}$ denote the integers of $\breve k$. Let $F$ denote the Frobenius automorphism of $\breve k$ over $k$. 

Let $G$ be a reductive group over $k$, which splits over $\breve k$. Let $T \subseteq B$ be a maximal torus and a Borel subgroup of $G$, such that $T$ splits and $B$ becomes rational over $\breve k$. Denote by $W$ the Weyl group of $T$ in $G$. Denote by $U$ resp. $U^-$ the unipotent radicals of $B$ resp. the opposite Borel subgroup and assume that $(T,U)$ is a Coxeter pair (see \S\ref{sec:notation}). Attached to $(T,U)$ there is a $p$-adic Deligne--Lusztig space, on geometric points given by
\[X_{T,U} = \{g \in G(\breve k) \colon g^{-1}F(g) \in (U^- \cap FU)(\breve k) \}.\] 
It admits a continuous action of $G(k) \times T(k)$ given by $(g,t) \colon x \mapsto gxt$. See \cite[\S7 and \S11]{Ivanov_DL_indrep} (and \S\ref{sec:recollections_pDL} below). Let $\theta \colon T(k) \to \overline\BQ_\ell^\times$ be a smooth character. The $\theta$-isotypic component $R_T^G(\theta)$ of the homology of $X_{T,U}$ is an object in the (derived) category of smooth $G(k)$-representations, cf. \cite{IvanovM}. The goal of this article is to further investigate properties of $R_T^G(\theta)$, extending and generalizing results from \cite{DI, CI_loopGLn}.

To describe our results we need more notation. The apartment of $T$ in the reduced Bruhat--Tits building of $G$ consists of one point. Bruhat--Tits theory attaches to this point a (connected) parahoric $\CO_k$-model $\CG$ of $G$. By \cite{Ivanov_Cox_orbits,Nie_23}, $X_{T,U}$ admits a decomposition $X_{T,U} = \coprod_{\gamma \in G(k)/\CG(\CO_k)} \gamma X_{T,U}^{\CG}$, where 
\[ X_{T,U}^{\CG} = \{g \in \CG(\CO_{\breve k}) \colon g^{-1}F(g) \in (\CU^- \cap F\CU)(\CO_{\breve k})\}\] 
is an affine $\overline \BF_q$-scheme (here $\CU$ denotes the closure of $U$ in $\CG$). 

Fix some $r\leq \infty$. We can regard $\CG(\CO_{\brk}/\varpi^r) = \BG(\overline \BF_q)$ as the geometric points of a perfect $\BF_q$-scheme $\BG = \BG_r$. This is done via the (truncated, if $r<\infty$) positive loop functor, see e.g. \cite[\S1.1]{Zhu_17} (or \cite[\S2]{DI}) for details. For a subscheme $H \subseteq G$, we denote by $\CH$ its closure in $\CG$ and by $\BH \subseteq \BG$ the corresponding subscheme of $\BG$. We denote by $F$ the geometric Frobenius of $\BG$, so that $\BG^F = \BG(\BF_q)$. Then $X_{T,U}^{\CG}$ is isomorphic to the inverse limit over $r$ of its truncations in each $\BG_r$. Each of these truncations is a perfectly smooth perfect $\overline\BF_q$-scheme, and up to an $\BA^n$-bundle (not affecting the cohomology), it equals
\begin{equation}\label{eq:X}
X = X^{\CG}_{T,U,r} = \{x \in \BG \colon x^{-1}F(x) \in F\BU \},
\end{equation}
Note that $X$ is equipped with the action of the finite group $\BG^F \times \BT^F$ given by $(g,t) \colon x \mapsto gxt$.

By these geometric considerations ($+\varepsilon$), $R_T^G(\theta)$ admits the following more explicit description (which might, for the purposes of this article, also be considered as a definition). Let $Z \subseteq G$ denote the center of $G$. For any $\BT^F$-module $M$, let $M[\chi]$ denote the $\chi$-isotypic subspace. Then, if $\theta|_{\CT(\CO_k)}$ factors through a character $\chi$ of $\BT^F$, then  
\[ 
R_T^G(\theta) = {\rm cInd}_{\CG(\CO_k)Z(k)}^{G(k)} H_c^\ast(X)[\chi], 
\]
where $H_c^\ast(X)[\theta] = \sum_{i\in \BZ} H_c^i(X,\overline\BQ_\ell)[\theta]$ is the $\ell$-adic equivariant Euler characteristic of $X$ (regarded as a virtual $\BG^F$-module), inflated to a $\CG(\CO_k)$-representation, and extended to $\CG(\CO_k)Z(k)$ in the unique way such that $Z(k)$ acts by $\theta|_{Z(k)}$. Our first main result concerns the representations in the cohomology of $X$. 


\begin{theorem}\label{thm:main}
Suppose that $q$ satisfies condition \eqref{eq:M} (this is always true when $q>5$). Then there exists a Coxeter pair $(T,U)$ such that
\[ \dim_{\overline\BQ_{\ell}} \Hom_{\BG^F}( H_c^\ast(X)[\chi], H_c^\ast(X)[\chi']) = \sharp\left\{w \in W_e^F; w(\chi) = \chi'\right\}\] 
for any two smooth characters $\chi, \chi'$ of $\BT^F$, where $W_e$ denotes the Weyl group of the special fiber of $\CT$ in the reductive quotient of the special fiber of $\CG$.
\end{theorem}

In particular, if $\{w \in W_e^F \colon w(\chi) = \chi\} = \{1\}$, then $H_c^\ast(X)[\chi]$ is up to sign an irreducible $\BG^F$-representation. Note that Theorem \ref{thm:main} generalizes \cite[Theorem 3.2.3]{DI} and \cite[Theorem 4.1]{CI_loopGLn}. 

\begin{remark}
    Recently, under a mild condition on $p$, Chan \cite{Chan24} shows by a different approach that the inner product formula holds in a much more general case, which in particular includes the case that $T$ is elliptic. 
\end{remark}

Our second main result concerns the cuspidality of the compactly induced $G(k)$-representation $R_T^G(\theta)$. It generalizes \cite[Theorem 6.1]{CI_loopGLn}.

\begin{theorem}\label{thm:cusp}
Assume that $G$ is unramified and that $q$ satisfies condition \eqref{eq:M}. Let $\theta \colon T(k) \to \overline\BQ_\ell^\times$ be smooth with trivial stabilizer in $W^F$. Then $R_T^G(\theta)$ is up to sign a finite direct sum of irreducible supercuspidal representations of $G(k)$.
\end{theorem}

Some comments on our results are in order. First, we explain why ``it suffices'' to establish Theorem \ref{thm:main} for a single Coxeter pair $(T,U)$. Ultimately, we are interested in the smooth $G(k)$-representation $R_T^G(\theta)$. By \cite[Corollary 7.25, Lemma 11.3]{Ivanov_DL_indrep}, $X_{T,U}$ are mutually $G(k) \times T(k)$-equivariantly isomorphic, when $(T,U)$ varies through all Coxeter pairs $(T,U)$ with a fixed $T$.\footnote{This is not clear for the schemes $X_{T,U,r}$ at least if $G$ is not quasi-split (for the quasi-split case, see  \cite[Corollary 4.1.4]{DI}).} Thus, $R_T^G(\theta)$ is independent of the choice of $U$. So, it suffices to know the statement of Theorem \ref{thm:main} for at least one Coxeter pair. In fact, our proof shows that for many groups $G$ Theorem \ref{thm:main} holds for all pairs $(T,U)$, see Remark \ref{rem:bipartite_Coxeterelements}.

\smallskip

Next, we explain why the condition on $q$ in Theorems \ref{thm:main} and \ref{thm:cusp} is very mild, so that the theorems even gives rise to new supercuspidal representations of $G(k)$. Recall that by the work of Yu and Kaletha \cite{Yu_01,Kaletha_19}, one can attach a supercuspidal irreducible $G(k)$-representation $\pi_{(S,\theta)}$ to any regular elliptic pair $(S,\theta)$ consisting of a maximal elliptic torus $S \subseteq G$ and a sufficiently nice smooth character $\theta \colon S(k) \rightarrow \overline\BQ_\ell^\times$. A crucial point for this to work is the existence of a Howe factorization of $\theta$, cf. \cite[\S3.6]{Kaletha_19}. However, not all characters admit a Howe factorization, when the residue characteristic $p$ is small and $G$ is not an inner form of $\GL_n$. 

For instance, if $p \in \{2,3,5\}$, there exist many examples of pairs $(T,\theta)$ with $T$ unramified Coxeter (hence covered by Theorem \ref{thm:main} when $q$ satisfied condition \eqref{eq:M} -- in particular, whenever $q>5$) such that ${\rm Stab}_{W_e^F}(\theta) = \{1\}$, but $\theta$ does not admit a Howe factorization. For examples of $(T,\theta)$ not admitting a Howe factorization we refer to the forthcoming work of Fintzen--Schwein \cite{FintzenSchwein}, where an algebraic approach to the extension of Yu's construction is pursued. As mentioned in \cite{ChanO_21}, since ${\rm Stab}_{W_e^F}(\theta) = \{1\}$ one should expect an irreducible supercuspidal $G(k)$-representation attached to $(T,\theta)$, but Yu's construction does not apply as there is no Howe factorization. The point is now that our cohomological construction does not require any condition on $p$, but only a mild one on $q$. In particular, there are many examples of $k,G,T,\theta$ such that $\pm H^\ast_c(X)[\theta|_{\CT(\CO_k)}]$ is an irreducible $\CG(\CO_k)$-representation, which does not appear in Yu's construction. Moreover, Theorem \ref{thm:cusp} implies that its induction to $G(k)$ is supercuspidal.
 

\subsection*{Acknowledgements}
The first author is grateful to Jessica Fintzen and David Schwein for explaining their results on characters without Howe decomposition.
The first author gratefully acknowledges the support of the German Research Foundation (DFG) via the Heisenberg program (grant nr. 462505253).

\section{Preparations}

\subsection{More notation}\label{sec:notation}

We use the notation from the introduction. Moreover, we denote by $N_G(T)$ the normalizer of $T_{\breve k}$ in $G_{\breve k}$, so that $W = N_G(T)/T$ is the Weyl group of $T$, by $X^\ast(T)$ (resp. $X_\ast(T)$) the group of characters (resp. cocharacters) of $T_{\breve k}$ and by $\langle\cdot,\cdot \rangle \colon X^\ast(T) \times X_\ast(T) \to \BZ$ the natural pairing. We write $\Phi$ for the root system of $T_{\breve k}$ in $G_{\breve k}$, $\Phi^+$ for the subset of positive roots determined by $B$, and $\Delta \subseteq \Phi^+$ for the subset of positive simple roots. We write $S \subseteq W$ for the corresponding set of simple reflections.

Let $c \in W$ be the unique element such that $FB = {}^c B$. Then for any lift $\dot c$ of $c$, $\Ad({\dot c})\i \circ F: G(\breve k) \to G(\breve k)$ fixes the pinning $(T,B)$, hence defines automorphisms, denoted by $\sigma$, of the based root system $\Delta \subseteq \Phi$ and of the Coxeter system $(W,S)$. Note that $\sigma$ does not depend on the choice of the lift $\dot c$. We call $(T,B)$ (or $(T,U)$) a \emph{Coxeter pair} if $c$ is a $\sigma$-Coxeter element in the Coxeter triple $(W,S,\sigma)$, that is, if a(ny) reduced expression of $c$ contains precisely one element from each $\sigma$-orbit on $S$. \emph{Moreover, we assume until the end of \S\ref{sec:proof_Coxeter} that $c$ is $\sigma$-Coxeter, and hence $(T,U)$ is a Coxeter pair.}

Except for $G$, $\CG$ and their subgroups (which are defined over $k,\breve k$ resp. $\CO_k,\CO_{\breve k}$), all schemes appearing below are perfect schemes perfectly of finite presentation and perfectly smooth over $\overline\BF_q$. For a review of perfect geometry we refer to \cite[Appendix A]{Zhu_17}. We freely make use of the $6$-functor formalism of \'etale cohomology for such schemes with $\overline\BQ_\ell$-coefficients. Moreover, we fix a prime number $\ell \neq p$, and for a perfect $\overline\BF_q$-scheme we denote by $H^\ast(Y) = H_c^\ast(Y, \overline\BQ_\ell)$ its $\ell$-adic \'etale cohomology with compact support. 


\subsection{Pinning}\label{sec:pinning}
We may express the action of the Frobenius $F$ on $X_\ast(T)_{\BQ}$ as $F = \mu c \sigma \colon x \mapsto \mu + c\sigma(x)$ for some $\mu \in X_\ast(T)$. There is a unique point $e \in \BQ\Phi^\vee$ such that $F(e) \in e + X_\ast(Z)_\BQ$, or equivalently, $\mu + c\sigma(e) - e \in X_\ast(Z)_\BQ$. Let 
\[
\Phi_e = \{\a \in \Phi; \<\a, e\> \in \BZ\}.
\] 
We denote by $\D_e$ the set of simple roots of $\Phi_e^+ = \Phi_e \cap \Phi^+$. Let $W_e \subseteq W$ be the Weyl group of $\Phi_e$. Note that $\CG$ from the introduction is the parahoric model attached to the image of $e$ in the reduced building of $G$, and that $\Phi_e$ (resp. $W_e$) is the root system (resp. Weyl group) of the reductive quotient of the special fiber of $\CG$. 

Also, note that the action of $F$ on $W$ agrees with ${\rm Ad}(c) \circ \sigma$; we denote it by $F = c\sigma \colon W \to W$. This action stabilizes $W_e \subseteq W$. Finally, for an element $w \in W_e$ we denote by $\dot w \in \BG(\overline\BF_q)$ an arbitrary (fixed) lift of $w$.

\subsection{A condition on $q$}
Let $\o_\a^\vee$ denotes the fundamental coweight of $\a \in \D$. For a $\s$-orbit $\CO \subseteq \D$ of simple roots, we set $\o_\CO^\vee = \sum_{\a \in \CO} \o_\a^\vee$, where $\o_\a^\vee$ denotes the fundamental coweight of $\a \in \D$. 
We prove our main result under the following condition on $q$:
\begin{equation}\label{eq:M}
q > M = \max\{\<\g, \o_\CO^\vee\>; \g \in \Phi^+, \CO \in \D / \<\s\>\}.
\end{equation}
Note that $M$ only depends on the (relative) Dynkin diagram $\Delta$ of the quasi-split inner form of $G$ over $k$. If $\Delta$ is connected then $M$ takes the following values: $M=1$ for type $A_n$; $M=2$ for types $B_n, C_n, D_n, {}^2A_n, {}^2D_n$; $M=3$ for types $G_2, E_6, {}^3D_4$; $M=4$ for types $F_4, E_7, {}^2E_6$; $M=6$ for type $E_8$. If the quasi-split inner form of $G$ is split, then $M$ is the same as in \cite[\S2.7]{DI}, and it differs otherwise. Just as in \cite[\S2.7]{DI}, for arbitrary $G$ the constant $M$ equals the maximum of the values of $M$ over all connected components of the Dynkin diagram of $G_{\breve k}$ (equipped with the smallest power of $\sigma$ fixing the connected component). In particular, \eqref{eq:M} holds whenever $q>5$.

\subsection{A Coxeter element in $W_e$}\label{subsec: sequence} It turns out that $c$ determines a (twisted) Coxeter element of $W_e$. Write $c = s_{\a_1} \cdots s_{\a_r}$, where $\{\a_1, \dots, \a_r\} \subseteq \D$ is a set of representatives of $\s$-orbits of $\D$.

Let $I = (i_1 < i_2 < \cdots < i_m)$ be a subsequence of $[r]:= (1 < 2 < \cdots < r)$, and let $I' = (j_1 < j_r < \cdots < j_{r-m})$ be the complement sequence of $I$ in $[r]$. We define \begin{align*} \s_{I, c} &= s_{\a_{i_1}} s_{\a_{i_2}} \cdots s_{\a_{i_m}} \s; \\ c_I &= s_{\b_{j_1}} s_{\b_{j_2}} \cdots s_{\b_{j_{r-m}}}; \\ \D_{I, c} &= \{\b_{j_l}; 1 \le l \le r-m\} \end{align*} where $\b_{j_l} = s_{\a_{i_1}} s_{\a_{i_2}} \cdots s_{\a_{i_t}}(\a_{j_l})$ with $1 \le t \le m-1$ such that $i_t < j_l < i_{t+1}$. By definition, $c \s = c_I \s_{I, c}$.


\begin{theorem} \label{Coxeter}
    Let $c$, $\mu$ and $e = e_{\mu, c}$ be as in \S\ref{sec:pinning}. Then there exist a sequence $I = I_{\mu, c}$ of $1 < 2 < \cdots < r$ such that 
    
    (1) $\s_{I, c}(\D_e) = \D_e$; 
    
    (2) $\D_{I, c} \subseteq \D_e$ is a representative set of $\s_{I, c}$-orbits of $\D_e$; 
    
    (3) $\s_{I, c}^i = 1$ if and only if $\s_{I, c}^i$ fixes each root of $\D_e$. 

    In particular, $c_I$ is a $\s_{I, c}$-Coxeter element of $W_e$.
\end{theorem}

This theorem is proven in \S\ref{sec:proof_Coxeter}.

\subsection{Support} For $\a \in \Phi$ we denote by $\supp(\a) \subseteq \D$ the minimal subset whose linear span contains $\a$. For a subset $C \subseteq \Phi$ we set $\supp(C) = \cup_{\a \in C} \supp(\a)$. For $w \in W$
we denote by $\supp(w)$ the set of simple reflections which appear in some/any reduced expression of $w$.
\begin{lemma} \label{stable}
    Let $C \subseteq \Phi$ be a $c\s$-orbit. Then $\supp(C)$ is $\s$-stable.
\end{lemma}
\begin{proof}
   Let $c = s_{\a_1} \cdots s_{\a_r}$ be as in \S\ref{subsec: sequence}. Let $\a \in \supp(\g)$ for some $\g \in C$. It suffices to show that the $\s$-orbit $\CO$ of $\a$ is contained in $\supp(C)$. Set $\d = \sharp \CO$. Let $1 \le j \le r$ be the unique integer such that $\a_j \in \CO$. Let $0 \le i_0 \le \d-1$ such that \[\a, \s\i(\a), \dots, \s^{1-i_0}(\a) \neq \a_j \text{ and } \s^{-i_0}(\a) = \a_j.\] Then one checks that $(c\s)^{-i_0} = \s^{-i_0} w$ for some $w \in W$ such that $\supp(w) \subseteq \D - \{\a\}$. Hence $\a \in \supp(w(\g))$ and $\a_j = \s^{-i_0}(\a) \in \supp(\s^{-i_0} w(\g)) = \supp((c\s)^{-i_0}(\g))$. So we can assume that $\a = \a_j$. Let $0 \le i \le \d-1$. Note that $(c\s)^i = u_i \s^i$ for some $u_i \in W$ with $\supp(u_i) \subseteq \D - \{\s^i(\a_j)\}$. It follows that $\s^i(\a) \in \supp((c\s)^i(\g))$. So the statement follows.
\end{proof}

\begin{proposition} \label{conn}
    Let $C$ be a $c\s$-orbit of $\Phi$. Then $\supp(C) = \cup_{i \in \BZ} \s^i(H)$, where $H$ is a connected component of $\D$.
\end{proposition}
\begin{proof}
Without loss of generality we may assume that $\D = \cup_{i \in \BZ} \s^i(H)$. We argue by induction on $\sharp \D$. Assume the statement is false. Let $c = s_{\a_1} \cdots s_{\a_r}$ be as in \S\ref{subsec: sequence}. By Lemma \ref{stable} there exists $1 \le j \le r$ such that $C \subseteq \Phi_K$, where $K = \D - \CO$ and $\CO$ is the $\s$-orbit of $\a_j$. By replacing $c$ with its $W_K$-$\s$-conjugate $s_{\a_j} \cdots s_{\a_r} \s(s_{\a_1} \cdots s_{\a_{j-1}})$, we can assume that $j = 1$ and $\a_1 \in \CO$. Let $c' = s_{\a_1} c$, which is a $\s$-Coxeter element of $W_K$. As $C \subseteq \Phi_K$, $C$ is also a $c'\s$-orbit of $\Phi_K$. By induction hypothesis we have $\supp(C) = \cup_{i \in \BZ} \s^i(D)$, where $D$ is a connected component of $H - \{\a_1\}$. As $H$ is connected, there exists $\g \in C$ and $\b \in \supp(\g)$ such that \[0 > \<\a_1, \b^\vee\> \ge \<\a_1, \g^\vee\>.\] Then we have $\s\i(\a_1) \in \supp((c\s)\i(\g))$, contradicting that $C \subseteq \Phi_K$. The proof is finished.     
\end{proof}

\subsection{A condition on the $\sigma$-Coxeter element} \label{set-up}
Let $c$, $\mu$, $e = e_{\mu, c}$, $I = I_{\mu, c}$, $c_I$, $\s_I = \s_{I, c}$ and $\D_I = \D_{I, c} \subseteq \D_e$ be as in Theorem \ref{Coxeter}. Denote by $\ell \colon W \to \BZ_{\geq 0}$ (resp. $\ell_e: W_e \to \BZ_{\ge 0}$) the length function associated to the set $\Delta$ (resp. $\D_e$) of simple roots. Let $w_0$ and $w_e$ be the longest elements of $W$ and $W_e$ respectively. We consider the following condition on $c$, or, equivalently, on the pair $(T,U)$:
\begin{equation}\label{eq:ast} \tag{$\ast$} \text{There exists $N \in \BZ_{\ge 1}$ such that }\, (c\s)^N = w_0 \s^N, \quad N \ell(c) = \ell(w_0).
\end{equation}

\begin{remark}\label{rem:bipartite_Coxeterelements} If $\Delta$ is connected, then there always exists a $\sigma$-Coxeter element $c \in W$ satisfying \eqref{eq:ast}, see \cite[Chap. V, Prop. 6.2]{Bourbaki_68}. Moreover, if the Coxeter number of $G$ is even, then any $c$ satisfies this condition.
\end{remark}

\begin{lemma} \label{ast}
Suppose $c$ satisfies condition \eqref{eq:ast}. Then $(c_I \s_I)^N = w_e \s_I^N$ and $N \ell_e(c_I) = \ell_e(w_e)$.
\end{lemma}
\begin{proof}
    By Theorem \ref{Coxeter}, $c_I \s_I = c \s$ and $\s_I(\D_e) = \D_e$. As $(c_I \s_I)^N = w_0 \s^N$, it follows that $(c_I \s_I)^N$ sends $\Phi_e^+$ to $-\Phi_e^+$, that is, $(c_I \s_I)^N = w_e \s_I^N$.

    It remains to show $\ell_e((c_I \s_I)^{i+1}) = \ell_e((c_I \s_I)^i) + \ell_e(c_I \s_I)$ for $1 \le i \le N-1$. Indeed, this is equivalent to that for any $\a \in \Phi_e^+$ with $(c_I \s_I)\i(\a) < 0$ we have $(c_I \s_I)^i(\a) > 0$. This statement follows from that $c_I \s_I = c\s$ and $\ell((c\s)^{i+1}) = \ell((c\s)^i) + \ell(c\s)$ for $1 \le i \le N-1$.
\end{proof}

For $w \in W$ we denote by $\supp(w)$ the set of simple reflections in $\D$ that appears in some/any reduced expression of $w$. For $u \in W_e$, we can define $\supp_{\D_e}(u) \subseteq \D_e$ in a similar way.
\begin{corollary} \label{proper}
    Suppose $c$ satisfies condition \eqref{eq:ast}. Let $K \subsetneq \D_e$ be a proper $\s_I$-stable subset. Then there exists a proper $\s$-stable subset $J \subsetneq \D$ such that $\s_I \in W_J \s$ and $w_e W_K \subseteq w_0 W_J$.
\end{corollary}
\begin{proof}
    Let notation be as in \S \ref{subsec: sequence}. As $\D_I = \{\b_j; j \in I'\}$ with $I' = [r] - I$ is a representative set of $\D_e$, there exists $i \in I'$ such that $\b_i \notin K$. Let $J = \D - \CO_i$, where $\CO_i$ is the $\s$-orbit of $\a_i$. By construction, $\supp(s) \subseteq J$ for $s \in K$ and $\supp(\s_I \s\i) \subseteq J$. By Lemma \ref{ast} we have \[w_e = (c_I \s_I)^N \s_I^{-N} = (c \s)^N \s_I^{-N} = w_0 \s^N \s_I^{-N} \subseteq w_0 W_J.\] Thus $w_e W_K \subseteq w_0 W_J$ as desired.    
\end{proof}

\begin{lemma} \label{partial}
    Let $K_1, K_2 \subseteq \D_e$ be two $\s_I$-stable subsets. Let $c_1$ and $c_2$ be two $\s_I$-Coxeter elements of $W_{K_1}$ and $W_{K_2}$ respectively. Let $w \in W_e$ such that $c_1 \s_I(w) = w c_2$. Then there exists $x \in {}^{K_1} W_e {}^{K_2}$ such that  ${}^x K_2 = K_1$ and  $w \in x W_{K_2}$.
\end{lemma}
\begin{proof}
By symmetry we may assume $\sharp K_1 \le \sharp K_2$.
Let $x \in {}^{K_1} W_e$ such that $w \in W_{K_1} x$. Then there exists $c_2' \leq c_2$ such that $x c_2 \in W_{K_1} x c_2'$ and $x c_2' \in {}^{K_1} W_e$. Hence we have $\s_I(x) = x c_2'$. Note that $c_2'$ is a partial $\s_I$-Coxeter element, which is of minimal length (in the sense of $\ell_e$) in its $\s_I$-conjugacy class. Thus $c_2' = 1$, $x = \s_I(x)$ and $x(\supp_{\D_e}(c_2)) \subseteq K_1$, which implies that $x(K_2) \subseteq K_1$. Hence $x(K_2) = K_1$ since $\sharp K_1 \leq \sharp K_2$. Thus $x \in {}^{K_1} W_e {}^{K_2}$ as desired.
\end{proof}

\section{Cohomology of $X$}

Recall the scheme $X$ from \eqref{eq:X} equipped with $\BG^F \times \BT^F$-action. 


\subsection{The schemes $\Sigma^i$}
Let $i \in \BZ$. We define \[\Sigma^i = \{(x, x', y) \in F\BU \times F^{i+1} \BU \times \BG; x F(y) = y x'\}.\] 
Let $\pr_3: \Sigma^i \to \BG$ be the natural projection. There is a locally closed decomposition \[\Sigma^i = \bigsqcup_{w \in W_e} \Sigma_w^i,\] where $\Sigma_w^i = \pr_3\i(\BU w \BT \BG^1 F^i\BU)$. 

The group $\BT^F \times \BT^F$ acts on $\Sigma^i$ and on each of the pieces $\Sigma^i_w$ by \[(t, t'): (x, x', y) \mapsto (t x t\i, t' x' {t'}\i, t y {t'}\i).\]
As in \cite[p.137]{DeligneL_76} there is a $\BT^F \times \BT^F$-equivariant isomorphism $X \times X/\BG^F \stackrel{\sim}{\to} \Sigma^0$, and for characters $\chi, \chi'$ of $\BT^F$ we have 
\[\dim_{\overline{\BQ}_\ell} \Hom_{\BG^F}(H_*(X)[\chi'], H_*(X)[\chi]) = \dim H_*(\Sigma^0)_{\chi', \chi\i},\] where $H_*(\Sigma^0)_{\chi', \chi}$ is the corresponding isotropic subspace of $H_c^*(\Sigma^0)$. 

Let $Z \subseteq G$ denote the centre of $G$ and consider the embedding $z \mapsto (z,z^{-1}) \colon Z \to T \times T$. Then the above $\BT^F \times \BT^F$-action on $\Sigma^i$ factors through an action of the quotient $\BT^F \times^{\BZ^F} \BT^F$. This latter action extends to the action of $\BT^F \times^{\BZ^F} \BT^F \subseteq (\BT \times^{\BZ} \BT)^F$ on $\Sigma^i$ (and $\Sigma^i_w$ for $w \in W_e$) given by the same formula. By the discussion in \cite[\S4.2]{DI} which applies in our more general setting, Theorem \ref{thm:main} follows from the next result.

\begin{theorem}\label{thm:main2}
Suppose that $q$ satisfies condition \eqref{eq:M}. Then there exists a Coxeter pair $(T,U)$ such that
\[H_\ast(\Sigma_w^0) = \begin{cases}  H_\ast((\dot w\BT)^{c\sigma}) & \text{if $w \in W_e^{c\sigma}$,} \\ \{0\} & \text{otherwise.} \end{cases}\]
as virtual $(\BT \times^{\BZ} \BT)^F$-modules.
\end{theorem}

As a first step towards the proof of Theorem \ref{thm:main2} we observe that the whole discussion of \cite[\S4.3]{DI} applies \emph{mutatis mutandis} in our setting. Thus it suffices to prove Theorem \ref{thm:main2} in the case that $\Delta$ is connected. In particular, there exists some $c$ satisfying condition \eqref{eq:ast}, cf. Remark \ref{rem:bipartite_Coxeterelements}. Now Theorem \ref{thm:main2} follows from Corollary \ref{sum} and Proposition \ref{prop:final_prop} below.

\subsection{An extension of action}
Let $w \in W_e$. We set $K_{w, i} = {}^{w\i}U^- \cap F^i U^-$. Define \[\hat\Sigma_w^i = \{(\tx, \tx', y_1, \t, z, y_2) \in F\BU \times F^{i+1}\BU \times \BU \times \dw\BT \times \BK_{w, i}^1 \times F^i \BU; \tx F(\t z) = y_1 \t z y_2 \tx'\}.\] We define an action of $\BT^F \times \BT^F$ on $\hat\Sigma_w^i$ by \[(t, t'): (\tx, \tx', y_1, \t, z, y_2) \mapsto (t \tx t\i, t' \tx' {t'}\i, t y_1 t\i, t \t {t'}\i, t' z {t'}\i, t' y_2 {t'}\i).\] Then there is an $\BT^F \times \BT^F$-equivariant affine space bundle \[\pi_w^i: \hat\Sigma_w^i \to \Sigma_w^i, \quad (\tx, \tx', y_1, \t, z, y_2) \mapsto (\tx F(y_1)\i, \tx' F(y_2), y_1 \t z y_2).\]

Let $\chi \in X_*(T)$ which centralizes $K_{w, i}$. Define \[H_{w, \chi} = \{(t, t') \in \BT \times \BT; w\i t\i F(t) w = {t'}\i F(t') \in \Im(\chi)\}.\] Then $H_{w, \chi}$ acts on $\hat\Sigma_w^i$ by \[(t, t'): (\tx, \tx', y_1, \t, z, y_2) \mapsto ({}^{F(t)} \tx, {}^{F(t')} \tx', {}^{F(t)} y_1, t \t {t'}\i, {}^{t'} z, {}^{F(t')} y_2).\]

\begin{lemma} \label{unstable}
    Let $w \in W_e \setminus W_e^{c\s}$ such that $\Sigma_w^i \neq \emptyset$. Then there exists a proper subset $K = \s_I(K) \subsetneq \D_e$ such that $w (c_I \s_I)^i \s_I^{-i} \in w_e W_K$.
\end{lemma}
\begin{proof}
    Let $w_i = w (c_I \s_I)^i \s_I^{-i} \in W_e$. By assumption we have \[c\s \BB w (c\s)^i \BB \BG^1 (c\s)^{-i-1} \cap \BB w(c\s)^i \BB \BG^1 c\s \BB (c\s)^{-i-1} \neq \emptyset.\] As $c_I \s_I = c\s$, this implies that \[c_I \s_I \BB_1 w (c_I \s_I)^i \BB_1 \cap \BB_1 w (c_I \s_I)^i \BB_1 c_I\s_I \neq \emptyset,\] that is, \[c_I \BB_1 \s_I(w_i) \BB_1 \cap \BB_1 w_i \BB_1 (\s_I)^i(c_I) \neq \emptyset.\] In particular there are $\s_I$-Coxeter elements $v' \leq_e c_I$ and $v \leq_e (\s_I)^i(c_I)$ of some $\s_I$-stable subsets $K'$ and $K$ of $\D_e$ respectively (one of which is a proper subset of $\D_e$ since $w \in W_e \setminus W_e^{c\s}$) such that $v' \s_I(w_i) = w_i v$ and \[\tag{a} \BB_1 w_i \BB_1 (\s_I)^i(c_I) \cap \BB_1 w_i v \BB_1 \neq \emptyset.\] Applying Lemma \ref{partial}, there exist $x = \s(x) \in {}^{K'} W_e {}^K$ such that $K' = {}^x K$ and $w_i \in x W_K$. Moreover, it follows from (a) that for any simple reflection $s \in \supp_{\D_e}(\s_I^i(c_I)) \setminus K$ we have $x s \in W_{K'} x$ or $x s \leq_e x$. The former is impossible since $s \notin W_K = x W_{K'} x\i$. So we have $x s \leq_e x$. Moreover, as $x s x\i \notin W_{K'}$ we have $w_{K'} x s \leq_e w_{K'} x = x w_K$, where $w_K$ and $w_{K'}$ are the maximal elements of $W_K$ and $w_{K'}$ respectively. As $x w_K$ is $\s_I$-stable, we have $x w_K s \leq_e x w_K$ for all $s \in \D_e$, that is, $x w_K = w_e$. Hence $w_i \in w_e W_K$.
\end{proof}

Let $N_0 \in \BZ_{\ge 0}$ be the order of $c\s \in W \rtimes \<\s\>$. Define \[N_F^{F^{N_0}}: \BT \to \BT, \quad t \mapsto t F(t) \cdots F^{N_0 - 1}(t).\] 

\begin{lemma} \label{non-zero}
    Let $\chi \in X_*(T)$ and let $C$ be a $c\s$-orbit of $\Phi$. Assume $\chi$ is non-central on $C$ and $|\<\chi, \b\>| < q$ for $\b \in C$. Then $\sum_{i=0}^{N_0 - 1} q^i \<\g, (c\s)^i(\chi)\> \neq 0$ for $\g \in C$. In particular, the action of $\BG_m$ on $\BU_\g$ for $\g \in C$, via the morphism $N_F^{F^{N_0}} \circ \chi$, is nontrivial.
\end{lemma}
\begin{proof}
    By assumption, $|\<\g, (c\s)^i(\chi)\>| = |\<(c\s)^{-i}(\g), \chi\>| < q$ for $0 \le i \le N_0-1$, and there exists $0 \le i_0 \le N_0-1$ such that $\<(c\s)^{-i_0}(\g), \chi\> \neq 0$. Hence the statement follows.
\end{proof}

Let $\BG_m \subseteq \CO_\brk^\times$ be the Teichm\"{u}ller lift of the quotient map $\CO_\brk^\times \to \overline \BF_q^\times$.  Assume that $r \in \BZ_{\ge 1}$. 
\begin{lemma} \label{morphism}
Consider the homomorphism \[f_{w, \chi}: \BG_m \to \BT \times \BT, \quad x \mapsto (N_F^{F^{N_0}}({}^w\chi(x)), N_F^{F^{N_0}}(\chi(x))).\] Then $\Im(f_{w, \chi}) \subseteq H_{w, \chi}^\circ$.
\end{lemma}
\begin{proof}
    By definition. $F^{N_0}(\l(x)) = \l(x^{q^{N_0}})$ for $x \in \brk$. Hence \[N_F(\chi(x))\i F(N_F(\chi(x))) = \chi(x)\i F^{N_0}(\chi(x)) = \chi(x\i\s^{N_0}(x)).\] So the statement follows.    
\end{proof}

\subsection{Handling $\Sigma_w^0$ for $w \in W_e \setminus W_e^{c\sigma}$}
Let $i \in \BZ$. Following \cite[\S5]{DI} we define an isomorphism of varieties \[\a_i: \Sigma^i \to \Sigma^{i+1}, \quad (x, x', y) \mapsto (x, F(x'), y x').\] For $w, u \in W_e$ we define \begin{align*}
    Y_{w, u}^i &= \Sigma_w^i \cap (\a_i)\i(\Sigma_u^{i+1}); \\
    Z_{w, u}^{i+1} &= \a_i(\Sigma_w^i) \cap \Sigma_u^{i+1} = \a_i(Y_{w, u}^i).
\end{align*} Let $\hat Y_{w, u}^i = (\pi_w^i)\i(Y_{w, u}^i)$ and $\hat Z_{w, u}^i = (\pi_u^{i+1})\i(Z_{w, u}^{i+1})$.

\begin{lemma} \label{piece}
    Let $w, u \in W_e$. Let $\chi, \mu \in X_*(T)$ which centralizes $\BK_{w, i}$ and $\BK_{u, i+1}$ respectively. Then $H_{w, \chi}$ and $H_{u, \mu}$ preserve $\hat Y_{w, u}^i$ and $\hat Z_{w, u}^{i+1}$ respectively.
\end{lemma}
\begin{proof}
    This is proved in \cite[\S5]{DI}.
\end{proof}

\begin{proposition} \label{vanishing}
    Suppose that condition \eqref{eq:ast} holds and that $q$ satisfies condition \eqref{eq:M}. Let $i \in \BZ$. Then \[H_*(\hat Y_{w, u}^i) = H_*(Y_{w, u}^i) = H_*(Z_{w, u}^{i+1}) = H_*(\hat Z_{w, u}^{i+1}) = 0\] if $w$ or $u$ belongs to $W_e \setminus W_e^{c\sigma}$.
\end{proposition}
\begin{proof}
    Without loss of generality we can assume that $w \in W_e \setminus W_e^{c\sigma}$ and $\hat Y_{w, u}^i \neq \emptyset$. In particular, $\Sigma_w^i \neq \emptyset$. By Lemma \ref{unstable} and Corollary \ref{proper}, there are subsets $K = \s_I(K) \subsetneq \D_e$ and $J = \s(J) \subsetneq \D$ such that \[w (c\s)^i \in w_e W_K (\s_I)^i \subseteq w_e W_J \s^i = w_0 W_J \s^i.\] Thus \[K_{w, i} \subseteq {}^{w\i}(U^- \cap {}^{w (c\s)^i} U^-) \subseteq {}^{w\i w_0} M_J,\] where $M_J$ is the Levi subgroup generated by $T$ and $U_\g$ for $\g \in \Phi_J$. Let $\CO \in \D \setminus J$ be a $\s$-orbit. Then $W_J$ fixes $\o_\CO^\vee$, and $K_w \subseteq {}^{w\i w_0} M_J$ is centralized by \[\chi := w\i w_0(\o_\CO^\vee) = w\i w (c\s)^i \s^{-i} (\o_\CO^\vee) = (c\s)^i(\o_\CO^\vee).\] Moreover, $w(\chi) = w_0 \s^N(\o_\CO^\vee) = (c\s)^N(\o_{\CO}^\vee)$.

    Let $f_{w, \chi}: \BG_m \to H_{w, \chi}$ be the as in Lemma \ref{morphism}. In view of Lemma \ref{piece}, via $f_{w, \chi}$ the action of $H_{w, \chi}$ on $\hat Y_{w, u}^i$ induces an action of $\BG_m$ on $\hat Y_{w, u}^i$, which commutes with action of $\BT^F \times \BT^F$. Hence \[H_c^*(Y_{w, u}) = H_c^*(\hat Y_{w, u}^i) = H_c^*((\hat Y_{w, u}^i)^{\BG_m}),\] it suffices to show $(\hat Y_{w, u}^i)^{\BG_m} = \emptyset$. To this end, we can assume that $\D = \cup_{i \in \BZ} \s^i(H)$ for some/any connected component $H$ of $\D$. Then by Proposition \ref{conn}, $\chi, w(\chi) \in \{(c\s)^i(\o_\CO^\vee); i \in \BZ\}$ are non-central on each $c\s$-orbit of $\Phi$. As $q > M$, it follows from Lemma \ref{non-zero} that \[(\hat Y_{w, u}^i)^{\BG_m} \subseteq \{1\} \times \{1\} \times \{1\} \times \BT \times \{1\} \times \{1\}.\] As $w \in W_e \setminus W_e^{c\s}$, we deduce that $(\hat Y_{w, u}^i)^{\BG_m} = \emptyset$ as desired.
\end{proof}

\begin{corollary} \label{sum}
    Let $i \in \BZ$ and $w \in W_e$. If $w \in W_e \setminus W_e^{c\s}$ then $H_*(\Sigma_w^i) = 0$. Otherwise, \[H_*(\Sigma_w^i) = \sum_{u \in W_e^{c\s}} H_*(Y_{w, u}^i) = \sum_{u \in W_e^{c\s}} H_*(Z_{u, w}^i) = \sum_{u \in W_e^{c\s}} H_*(Y_{u, w}^{i-1}).\] 
\end{corollary}
\begin{proof}
    Note that $\Sigma_w^i = \sqcup_{u \in W_e} Y_{w, u}^i = \sqcup_{u \in W_e} Z_{u, w}^i$ and $Z_{u, w}^i \cong Y_{u, w}^{i-1}$. Then the statement follows from Proposition \ref{vanishing}.
\end{proof}

\subsection{Handling $\Sigma_w^0$ for $w \in W_e^{c\sigma}$}
\begin{lemma} \label{nonempty}
    Suppose that Condition \eqref{eq:ast} holds. Let $i \in \BZ$ and $w, u \in W_e^{c\s}$ such that $Y_{w, u}^i \neq \emptyset$. Then $w = u$ if either $\s_I \neq 1$ or $\s_I = 1$ and $w c_I^i \neq w_e$.
\end{lemma}
\begin{proof}
     By assumption, we have \[\BB_1 w (c_I \s_I)^i \BB_1 c_I \s_I \BB_1 (c_I \s_I)^{-i-1} \cap \BB_1 u (c_I \s_I)^{i+1} \BB_1 (c_I \s_I)^{-i-1} \neq \emptyset,\] that is, $\BB_1 w(c_I \s_I)^i \BB_1 c_I \s_I \BB_1 \cap \BB_1 u (c_I \s_I)^{i+1} \BB_1 \neq \emptyset$. Thus there exists $v \leq_e c_I$ such that $w(c_I \s_I)^i v \s_I = u(c_I \s_I)^{i+1}$. Note that $w, u \in W_e^{c\s} \subseteq \<c_I \s_I\>$. We have \[v \s_I = (c_I \s_I)^{-i} w\i u (c_I \s_I)^{i+1} = w\i u (c_I \s_I) \in \<c_I \s_I\>.\] In particular, it follows from Lemma \ref{ast} that $\ell_e(v)$ is divided by $\ell_e(c_I)$. 

     Assume that either $\s_I \neq 1$ or $\s_I = 1$ and $w c_I^i \neq w_e$. If $v \neq 1$, then $\ell_e(v) = \ell_e(c_I)$ since $1 \neq v \leq c_I$. Hence we have $v = c_I$ and $w = u$ as desired. Suppose $v = 1$. Then $c_I = u\i w \in W_e^{c\s}$, which means that $\s_I(c_I) = c_I$. Hence $\s_I = 1$ by Theorem \ref{Coxeter} (3). By assumption we have $\s_I = 1$ and $w \s_I^i \neq w_e$. As $v = 1$, we have $w c_I^i s < w c_I^i$ for all $s \in \supp_{\D_e}(c_I) = \D_e$, that is, $w c_I^i = w_e$, a contradiction. 
\end{proof}

\begin{theorem}[\cite{IvanovNie_24}, Theorem 3.1] \label{cross-section}
    The map $(u_1, u_2) \mapsto u_1\i u_2 F(u_1)$ gives an isomorphism \[\phi:(F\BU \cap \BU) \times (F\BU \cap \BU^-) \cong F\BU.\] In particular, $\phi$ restricts to an isomorphism \[(F\BU^1 \cap \BU) \times (F\BU \cap \BU^-) \cong  \BU^1 (F\BU \cap \BU^-).\]
\end{theorem}

For $i \in \BZ$ and $w \in W_e$ we define \[{}^\flat \Sigma_w^i = \{(x, x', y) \in (F\BU \cap \BU^-) \times F^i(F\BU \cap \BU^-) \times (\BB \dw \BG^1 F^i\BB); x F(y) = y x'\}.\]
\begin{lemma} \label{flat}
    The map $(x, x', y) \mapsto (x_2, x_2', x_1 y F^i(x_1')\i, x_1, x_1')$ gives an $\BT^F \times \BT^F$-equivariant isomorphism \[\Sigma_w^i \cong {}^\flat\Sigma_w^i \times (F\BU \cap \BU) \times (F\BU \cap \BU),\] where $(x_1, x_2) = \phi\i(x)$ and $(x_1', x_2') = \phi\i(x')$. In particular, $H_c^*(\Sigma_w^i) \cong H_c^*({}^\flat\Sigma_w^i)$.
\end{lemma}
\begin{proof}
    It follows by definition and Theorem \ref{cross-section}.
\end{proof}

\begin{lemma}\label{indetity-case}
   Suppose that $c$ satisfies condition \eqref{eq:ast}. Let $w = (c\s)^m \in W_e^{c\s}$ for some $m \in \BZ$. Then we have $H_*(\Sigma_w^{N-m}) = H_*(\dw \BT^F) = H_*({}^\flat\Sigma_w^{2N-m}) = H_*(\Sigma_w^{2N-m})$. 
\end{lemma}
\begin{proof}
    The first equality is proved in \cite{DI}. We show the last two equalities. Let $(x, x', y) \in {}^\flat\Sigma_w^{2N-m}$. By definition, \[y \in  \BG^1 \BB \dw (c\s)^{2N-m} \BB (c\s)^{m-2N} = \BU \BT \BU^{-, 1} \dw.\] So we may write $y = y_1 \t y_2 w$ uniquely with $y_1 \in \BU$, $\t \in \BT$ and $y_2 \in \BU^{-, 1}$.

Then the equality $x F(y) = y x'$ is equivalent to \[\t\i y_1\i x F(y_1) F(\t) = y_2 \dw x' \dw\i F(y_2\i) = y_2 x'' F(y_2\i),\] where $x'' = {}^\dw x' \in F\BU \cap \BU^-$ since $w = (c\s)^m$.

By Theorem \ref{cross-section}, the map $(g_1, g_2) \mapsto g_1\i g_2 F(g_1)$ gives isomorphisms  \begin{align*}\BU \times (F\BU \cap \BU^-) &\cong \BU (F\BU \cap \BU); \\ \BU^{-, 1} \times (F\BU \cap \BU^-) & \cong (F\BU \cap \BU^-) F\BU^{-, 1}. \end{align*} So we can make changes of variables $(x, x'', y_1, y_2) \mapsto (z_1, z_2, z_3, z_4)$, where \[(z_1, z_2, z_3, z_4) \in \BU \times F\BU \cap \BU^- \times \BU^{-,1} (F\BU \cap \BU^-) \times F\BU^{-, 1} \cap \BU\] such that $y_1\i x F(y_1) = z_1 z_2$ and $y_2 x'' F(y_2)\i = z_3 z_4$. Then we have \[\t\i z_1 z_2 F(\t) = {}^{\t\i}z_1 L(\t) {}^{F(\t)\i}z_2  = z_3 z_4,\] where $L(\t) = \t\i F(\t)$. As $z_4 \in \BU^1$ we can have \[{}^{F(\t)\i}z_2 z_4\i = h_+ h_0 h_- \in \BU \BT \BU^-,\] where $h_+ \in \BU^1$, $h_0 \in \BT^1$ and $h_- \in (F\BU \cap \BU^-) \BU^{-, 1} = F(\BU \BU^{-, 1}) \cap \BU^-$. Hence \[{}^{\t\i}z_1 {}^{L(\t)} h_+ L(\t) h_0 h_- = z_3.\] It follows that $z_1 = {}^{F(\t)}h_+\i$, $L(\t) = h_0\i$ and $z_3 = h_-$. Therefore, \[{}^\flat\Sigma_w^{2N-m} = \{(\t, z_2, z_4) \in \BT \times (F\BU \cap \BU^-) \times (F\BU^{-, 1} \cap \BU) ; L(\t) = \pr_0({}^{F(\t)\i} z_2 z_4\i)\},\] where $\pr_0: \BU^1 \BT \BU^- \to \BT$ is the natural projection. 

Note that $(t, t') \in \BT^F \times \BT^F$ acts on ${}^\flat\Sigma_w^i$ by $(\t, z_2, z_4) \mapsto (t \t w(t')\i, {}^t z_2, {}^{w(t')} z_4)$. Now we define and action of $s \in \BT$ on ${}^\flat\Sigma_w^i$ by $(\t, z_2, z_4) \mapsto (\t, {}^s z_2, {}^s z_4)$. Then the actions of $\BT$ and $\BT^F \times \BT^F$ commutes with each other. Thus, by Lemma \ref{flat} we have \[H_*(\Sigma_w^{2N-m}) = H_*({}^\flat \Sigma_w^{2N-m}) = H_*(({}^\flat\Sigma_w^{2N-m})^{\BT}) = H_*(\dw\BT^F)\] as desired.   
\end{proof}

\begin{proposition}\label{prop:final_prop}
    Suppose that Condition \eqref{eq:ast} holds and that $\Delta$ is connected. Then $H_c^*(\Sigma_w^0) = H_c^*(\dw \BT^F)$ for $w \in W_e^{c\sigma}$.
\end{proposition}
\begin{proof}
    Let $w \in W_e^{c\s}$. As $\Delta$ is connected, we may write $w = (c\s)^m$ for some $m \in \BZ$. By Corollary \ref{sum} we have \[\tag{a} H_*(\Sigma_w^i) = \sum_{u \in W_e^{c\s}} H_*(Y_{w, u}^i), \quad  H_*(\Sigma_w^{i+1}) = \sum_{u \in W_e^{c\s}} H_*(Y_{u, w}^i).\]

    First we assume $\s_I \neq 1$. By Lemma \ref{nonempty} for any $w', u' \in W_e^{c\sigma}$ we have $Y_{w', u'}^i \neq \emptyset$ if and only if $w' = u'$. It follows by (a) that \[H_*(\Sigma_w^i) = H_*(Y_{w, w}^i) = H_*(\Sigma_w^{i+1}).\] By Lemma \ref{indetity-case} we have $H_*(\Sigma_w^0) = H_*(\Sigma_w^{N-m}) = H_*(w \BT^F)$ as desired.

    Now we assume $\s_I = 1$. Let notation be as in Lemma \ref{ast}. We can assume that $w = c_I^m$ with $0 \le m \le 2N-1$. If $0 \le m \le N$, it follows from (a), Lemma \ref{nonempty} and Lemma \ref{indetity-case} that \[H_*(\Sigma_w^0) = H_*(\Sigma_w^1) = \cdots = H_*(\Sigma_w^{N-m}) = H_*(\dw \BT^F).\] If $N+1 \le m \le 2N-1$, similarly we have \[H_*(\Sigma_w^0) = H_*(\Sigma_w^1) = \cdots = H_*(\Sigma_w^{2N-m}) = H_*(\dw \BT^F).\] So the statement follows.     
\end{proof}

\section{Proof of Theorem \ref{Coxeter}}\label{sec:proof_Coxeter}
In this section, we fill in the proof for Theorem \ref{Coxeter}. First we show that it suffices to consider one particular Coxeter element.
\begin{lemma}\label{cyclic}
    Let $\a \in \{\a_1, \s\i(\a_r)\}$ such that $c' = s_\a c \s(s_\a)$. Suppose Theorem \ref{Coxeter} holds for $(\mu, c)$. Then it also holds for $(s_{\a}(\mu), c')$.
\end{lemma}
\begin{proof}
Let $\mu, e, I$ be as in Theorem \ref{Coxeter}. Let $e' = 
 e_{s_\a(\mu), c'} = s_\a(e)$ and $\Phi_{e'} = s_\a(\Phi_e)$. Assume that $I= (i_1 < \cdots < i_m)$. Without loss of generality we can assume $\a = \s\i(\a_r)$ and $c' = s_{\a_1'} s_{\a_2'} \cdots s_{\a_r'}$ with $\a_1' = \a_r$ and $\a_i' = \a_{i-1}$ for $2 \le i \le r$. 

First we assume that $r \in I$. Then $r = i_m$ and $\s_{I, c}(\a) < 0$, which means that $\a \notin \D_e = \s_{I, c}(\D_e)$. Thus $\Phi_{e'}^+ = s_{\a}(\Phi_e^+)$ since $\a \in \D$ is a simple root. In particular, $\D_{e'} = s_\a(\D_e)$. We take \[I' = (1< i_1+1 < i_2+1 < \cdots < i_{m-1}+1).\] Then $\s_{I', c'} = s_\a \s_{I, c} s_\a$, $c'_{I'} = s_\a c_I s_\a$ and the statement follows.

Now we assume that $r \notin I$. Then $\s_{I, c}(\a) \in \D_{I, c} \subseteq \D_e = \s_{I, c}(\D_e)$. Thus $\a \in \D_e$ and $\D_{e'} = \D_e$. We take \[I' = (i_1+1 < i_2+1 < \cdots < i_m+1).\] Then $\s_{I', c'} = \s_{I, c}$, $c_{I'}' = s_\a c_{I} \s_{I, c}(s_\a)$ and the statement also follows.    
\end{proof}

To finish the proof, we will take a particular $\s$-Coxeter element $c$ such that and verify the statement directly. Moreover, we can assume $\D$ is connected. 

Let $P$ be the coweight lattice of $\Phi$. If $\mu = 0 \in P / (1 - c\s) P$, then $\D_e = \D$ and the statement is trivial. So we may assume that $P / (1-c)P \neq \{0\}$, which excludes the types ${}^2 A_{n-1}$ ($n$ odd), $^2 D_n$, $^3D_4$, $E_8$, $^2E_6$, $F_4$, $G_2$. Then we will take a case-by-case analysis for the remaining types.

We adopt the labelling of Dynkin diagrams by positive integers as in \cite{Hu}. For $i \in \BZ_{\ge 1}$. let $s_i$ and $\o_i^\vee$ denote the corresponding simple reflection and fundamental coweight, respectively.

\

Case (1): $\D$ is of type $A_{n-1}$. Take $c = s_1 s_2 \cdots s_{n-1}$. Then we have $P/(1 - c\s)P = \{0, \o_1^\vee, \o_2^\vee, \dots \o_{n-1}^\vee\}$. Assume $\mu = \o_k^\vee$ with $k \in \BZ$. Let $m = {\rm{gcd}}(k, n) \in \BZ_{\ge 1}$. Then we take $I$ to be the complement of the sequence $I' = (n/m, 2n/m, \cdots, (m-1)n/m)$.

\

Case (2): $\D$ is of type ${}^2 A_{n-1}$ with $n \ge 4$ even. Take $c = s_1 s_2 \cdots s_{n/2}$. Then $P/(1 - c\s)P = \{0, \o_1^\vee\}$. Assume $\mu = \o_1^\vee$. Then we take $I = (n/2)$.

\

Case (3): $\D$ is of type $B_n$ with $n \ge 2$. Take $c = s_1 s_2 \cdots s_n$. Then $P/(1 - c\s)P = \{0, \o_1^\vee\}$. Assume $\mu = \o_1^\vee$. Then we take $I = (n)$.

\

Case (4): $\D$ is of type $C_n$ with $n \ge 3$. Take $c = s_1 s_2 \cdots s_n$. Then $P/(1 - c\s)P = \{0, \o_n^\vee\}$. Assume $\mu = \o_n^\vee$. Then we take $I = (1, 3, \dots, n - \frac{(-1)^n + 1}{2})$.

\

Case (5): $\D$ is of type $D_n$ with $n \ge 4$. Take $c = s_1 s_2 \cdots s_n$. Then $P/(1 - c\s)P = \{0, \o_1^\vee, \o_{n-1}^\vee, \o_n^\vee\}$. If $\mu = \o_1^\vee$, take $I = (n-1, n)$. It remains to handle the case $\mu = \o_{n-1}^\vee$ by symmetry. If $n$ is even, take $I = (1, 3, \dots, n-3, 4)$ if $4 \mid n$ and $I = (1, 3, \dots, n-3, n-1)$ if $4 \nmid n$. If $n$ is odd, take $I= (1, 3, \dots, n-4, \dots, n-2, n)$. 

\

Case (6): $\D$ is of type $E_6$. Take $c = s_1 s_3 s_4 s_2 s_5 s_6$. Then $P/(1 - c\s)P = \{0, \o_1^\vee, \o_6^\vee\}$. By symmetry we can assume $\mu = \o_1^\vee$. Then take $I = (1, 3, 5, 6)$.

\

Case (7): $\D$ is of type $E_7$. Take $c = s_7 s_6 s_5 s_4 s_2 s_3 s_1$. Then $P/(1 - c\s)P = \{0, \o_7^\vee\}$. If $\mu = \o_7^\vee$, take $I = (7, 5, 2)$.

\section{Quotients of the Coxeter variety}\label{sec:quotient_Cox_var}


The goal of the rest of the article is to prove Theorem \ref{thm:cusp}. Therefore, mainly following \cite[\S2]{Lusztig_76_Inv}, we investigate quotients of $p$-adic Deligne--Lusztig spaces of Coxeter type by the unipotent radical of a rational Borel subgroup resp. of a maximal parabolic subgroup. We apply this at the end of \S\ref{sec:quots_finite_level} to deduce a proof of Theorem \ref{thm:cusp}.

\subsection{Notation}

We keep the notation from the introduction and \S\ref{sec:notation}, except for the following important change: \emph{from now on we assume that $G$ is unramified and that the Borel subgroup $B \subseteq G$ is $k$-rational}. We denote by $w_0 \in W$ the longest element (relative to $S$). If $v \in W$ is given, then by $\dot v$ we mean an arbitrary lift of $v$ to $N_G(T)(\breve k)$.

For $b\in G(\breve k)$ and a subgroup $H \subseteq G$ we denote by $H_b(k)$ the $F$-centralizer of $b$ in $H(\breve k)$, that is $H_b(k) = \{h \in H \colon h^{-1}bF(h) = b\}$.

We use the setup from \cite{Ivanov_DL_indrep}. In particular, we denote by ${\rm Perf}$ the category of perfect $\overline\BF_q$-algebras. For a $\breve k$-scheme $X$ we write $LX$ for the loop space of $X$, i.e., the functor $LX \colon {\rm Perf} \to {\rm Sets}$, $R \mapsto X(\BW(R)[\varpi^{-1}])$, where $\BW(R)$ is the unique $\varpi$-adically complete and separated $\CO_k$-algebra in which $\varpi$ is not a zero divisor and which satisfies $\BW(R)/\varpi\BW(R) = R$ (see \cite[page 6]{Ivanov_DL_indrep} for details).


\subsection{Recollections on $p$-adic Deligne--Lusztig spaces}\label{sec:recollections_pDL}
To $w\in W$ and $b \in G(\breve k)$, \cite[Definition 7.3]{Ivanov_DL_indrep} attaches a $p$-adic Deligne--Lusztig space $X_w(b)$ equipped with a continuous $G_b(k)$-action ($G_b(k)$  is locally profinite and equals the group of $k$-points of an inner form of a Levi subgroup of $G$). The definition of $X_w(b)$ parallels the classical Deligne--Lusztig variety from \cite{DeligneL_76}. Formally, $X_w(b)$ is an arc-sheaf on the category of perfect $\overline\BF_q$-algebras; it is known to be ind-representatble in many cases. Moreover, for a lift $\dot w\in N_G(T)(\breve k)$ one has a pro-\'etale $T_w(k)$-torsor $\dot X_{\dot w}(b)$ over a clopen subset of $X_w(b)$ \cite[\S10]{Ivanov_DL_indrep}, where $T_w$ is the $\breve k$-split form of the torus $T$ given by twisting the Frobenius $F$ by $\Ad(w)$. 


To relate this with the previous part of the article, consider the case $w=c$ is a $\sigma$-Coxeter element and a lift $\dot c \in N_G(T)(\breve k)$. If $T' \subseteq B' \subseteq G$ is a $k$-rational torus unramified and of type $c$ and a Borel subgroup (rational over $\breve k$), such that $B',F(B')$ are in relative position $c$, then there is an $G_{\dot c}(k) \times T_c(k) \cong G_{\dot c}(k) \times T'(k)$-equivariant isomorphism $X_{T',U'} \cong \dot X_{\dot c}(\dot c)$, where $U'$ is the unipotent radical of $B'$. 

\subsection{Embedding into the big cell}\label{sec:pDL_embedding_into_big_cell_loop}

Let $s_1,s_2,\dots, s_n$ be a sequence of pairwise distinct elements in $S$. Let $w = s_1s_2\dots s_n \in W$. Let $\alpha_i \in \Delta \subseteq \Phi^+$ be the simple root corresponding to $s_i$. Fix an isomorphism $\psi_i \colon U_{\alpha_i} \cong \BG_a$ of the corresponding root subgroup with the additive group. We have the open subscheme $\BG_m\subseteq \BG_a$ and we put $U_{\alpha_i}^\ast := \psi_i^{-1}(\BG_m)$. 

\begin{lemma}[Proposition 2.2 of \cite{Lusztig_76_Inv}]\label{lm:Lus2_2_loop}
Let $\tau \in T(\breve k)$. As locally closed subvarieties of $G$, we have 
\[
(\tau U^-) \cap B w B = \{ \tau v_1 v_2 \dots v_n \colon  v_i \in (U_{-\alpha_i})^\ast \text{ for $1\leq i \leq n$} \}.
\]
In particular, $U^- \cap B w B \cong \prod_{i=1}^n \BG_m$.
\end{lemma}
\begin{proof}
As $B w B$ is stable under left multiplication by $\tau$, we may assume that $\tau  = 1$. 
In this case the proof of \cite[Prop.~2.2]{Lusztig_76_Inv} for reductive groups over $\overline\BF_q$ carries over to the present situation, using the geometric Bruhat decomposition for the split group $G_{\breve k}$.
\end{proof}

\begin{lemma}[Lemma 2.3 of \cite{Lusztig_76_Inv}]\label{lm:Lus2_3}
Suppose any $\sigma$-orbit on $S$ contains at most one $s_i$. Fix an algebraically closed field $\mathfrak{f} \supseteq \overline\BF_q$, and let $\tilde{\mathfrak{f}} = \BW(\mathfrak{f})[1/\varpi]$. Let $v \in W$, such that
\[
\dot v^{-1} u F(\dot v) = B(\tilde{\mathfrak{f}}) \dot w B(\tilde{\mathfrak{f}}) = (B \dot w B)(\tilde{\mathfrak{f}}),
\]
for some $u \in U(\tilde{\mathfrak{f}})$. Then $F(v) = v$ and $v(\alpha_i) \in \Phi^-$ for $1\leq i \leq n$.
\end{lemma}
\begin{proof}
The proofs of \cite[Lemmas 2.3 and 2.4]{Lusztig_76_Inv} carry over \emph{verbatim}.
\end{proof}

For each $v \in W$, $U \cap {}^v U^- \to B v B/B$, $u \mapsto u\dot v B$ is an isomorphism of $\breve k$-schemes. In particular, $L(BvB/B) \cong L(U \cap {}^v U^-)$ is an ind-scheme.
We can now show the analogue of \cite[Cor.~2.5]{Lusztig_76_Inv}. 

\begin{proposition}\label{prop:Coxeter_DLspace_factors_big_cell_loop}   
Suppose $b \in T(\breve k)$, and $w$ is a $\sigma$-Coxeter element. The natural inclusion $X_w(b) \hookrightarrow L(G/B)$ factors through the big cell $L(B w_0 B/B) \subseteq L(G/B)$. 
\end{proposition}
\begin{proof}
Let $g \in X_w(b)(R) \subseteq L(G/B)(R)$. We must show that $g\in L(B \dot w_0 B/B)(R)$. By \cite[Corollary 8.4]{Ivanov_DL_indrep}, we may replace $R$ by an arc-cover. Hence, by \cite[Corollary 6.4]{Ivanov_DL_indrep} we may assume that $g$ lifts to some $\dot g \in LG(R)$. It suffices to show that $\dot g \colon \Spec R \to LG$ factors through $L(Bw_0B) \subseteq LG$. Since $LG, L(Bw_0B)$ are perfect (hence reduced) ind-schemes, it suffices to do this on geometric points. Hence we may assume $R = \mathfrak{f}$ is an algebrically closed field; let $\tilde{\mathfrak{f}} = \BW(\mathfrak{f})[1/\varpi]$. Note that we have $\dot g^{-1} b\sigma(\dot g) \in L(B w B)(\mathfrak{f}) = (BwB)(\tilde{\mathfrak{f}})$. We may argue as in the proof of \cite[Corollary~2.5]{Lusztig_76_Inv}: as $\tilde{\mathfrak{f}}$ is a field, and $G$ is  split over $\tilde{\mathfrak{f}}$, $G(\tilde{\mathfrak{f}})$ admits a Bruhat decomposition. Thus there is some $v \in W$ such that $\dot g = u \dot v \lambda$ for some $\lambda \in B(\tilde{\mathfrak{f}})$, $u \in U(\tilde{\mathfrak{f}})$. By the preceding paragraph we deduce $\dot v^{-1}u^{-1} b F(u) F(\dot v) = \dot g^{-1}b F(\dot g) \in (B \dot w B)(\tilde{\mathfrak{f}})$. By assumption $b \in T(\breve k)$, and we deduce $\dot v^{-1}u' F(\dot v) \in (B \dot w B)(\tilde{\mathfrak{f}})$, where $u' = (u^b)^{-1}F(u) \in U(\tilde{\mathfrak{f}})$. Thus we may apply Lemma \ref{lm:Lus2_3} (in the same way as in the proof of \cite[Corollary~2.5]{Lusztig_76_Inv}) to deduce that $v = w_0$. This shows our claim. 
\end{proof}

\subsection{$U_b(k)$-quotient of $X_w(b)$}

We now show the analogue of \cite[Theorem~2.6 and Corollary 2.7]{Lusztig_76_Inv}. \emph{Suppose $w = s_1\dots s_n$ is a $\sigma$-Coxeter element and let $b \in T(\breve k)$.} Consider the morphism
\[
U_{-w_0(\alpha_1)}^\ast \times U_{-w_0(\alpha_2)}^\ast \times \dots \times U_{-w_0(\alpha_n)}^\ast \to bU, \quad (v_i)_{i=1}^n \mapsto bv_1 v_2 \dots v_n,
\]
with image the locally closed subscheme $b\cdot \prod_{i=1}^n U_{-w_0(\alpha_i)}^\ast \subseteq bU$. The product depends on the order of the factors. Define $X_w(b)'$ by the Cartesian diagram
\begin{equation}\label{eq:def_Xwbprime_loop}
\xymatrix{
X_w(b)' \ar[r] \ar[d] & L\left(b \cdot \prod_{i=1}^n U_{-w_0(\alpha_i)}^\ast \ar[d]\right) \ar@{=}[r] & b \cdot \prod_{i=1}^n LU_{-w_0(\alpha_i)}^\ast \\
LU \ar[r]^{u \mapsto u^{-1}bF(u)} & b\cdot LU, 
}
\end{equation}
where the lower map is well-defined as $b$ normalizes $LU$. 
Note that $B_b(k)$ acts on $X_w(b)'$ by left multiplication. 

\begin{lemma}\label{lm:LUrarbLU}
Let $b\in T(\breve k)$. Then $u \mapsto u^{-1} b F(u) \colon LU \to b\cdot LU$ is $U_b(k)$-torsor for the pro-\'etale topology. The upper map in \eqref{eq:def_Xwbprime_loop} is a pro-\'etale $U_b(k)$-torsor. Moreover, $U_b(\breve k)$ is the group generated by all $U_\alpha(\breve k)$ with $\alpha \in \Phi^+$, $\ord_{\varpi}(\alpha(b)) = 0$. 
\end{lemma}
\begin{proof}
The second claim follows from the first. For the first claim, it is enough to show surjectivity of ${\rm La}_b \colon LU \to LU$, $u \mapsto (u^b)^{-1} F(u)$ for the pro-\'etale topology. Let the \emph{hight} of a root $\alpha \in \Phi^+$ be the smallest integer ${\rm ht}(\alpha) \geq 1$, such that $\alpha$ can be written as a sum of ${\rm ht}(\alpha)$ simple roots. For $i\geq 1$, let $U_{\leq i}$ be the quotient of $LU$ by the subsheaf generated by all $LU_{\alpha}$ with $\alpha \in \Phi^+$, ${\rm ht}(\alpha) > i$. Let $U_{= i}$ be the subsheaf of $U_{\leq i}$ generated by $LU_\alpha$ with ${\rm ht}(\alpha) = i$. Then $U_{=i} = \ker(U_{\leq i} \twoheadrightarrow U_{\leq i-1})$ is central in $U_{\leq i}$. Using this and induction on $i$, it suffices to show that ${\rm La}_b$ induces a surjection $U_{=i} \to U_{=i}$. But $U_i \cong \prod_{\alpha \colon {\rm ht}(\alpha) = i} L\BG_a$, ${\rm La}_b$ stabilize all factors, and the result follows from Lemma \ref{lm:LGa_torsor_over_LGa} below.
\end{proof}

\begin{lemma}\label{lm:LGa_torsor_over_LGa}
Let $\beta \in \breve k^\times$. Consider the map ${\rm La}_{\beta,\varphi} \colon L\BG_a \to L\BG_a$, $x \mapsto \beta x - \varphi(x)$. If $\ord_\varpi(\beta) \neq 0$, ${\rm La}_{\beta,\varphi}$ is an isomorphism. If $\ord_\varpi(\beta) = 0$, ${\rm La}_{\beta,\varphi}$ is a pro-\'etale torsor under the locally profinite group $k$. 
\end{lemma}
\begin{proof}
It suffices to show that if $R \in {\rm Perf}_{\overline\BF_q}$ is strictly henselian, and $x \in \BW(R)[1/\varpi]$, then there exists an (unique if $\ord_\varpi(\beta) \neq 0$) element $y \in \BW(R)[1/\varpi]$ with $\beta y - \varphi(y) = x$. This reduces to an explicit computation, using the (uniquely determined) $\varpi$-adic expansions $x = \sum_i [x_i] \varpi^i$, $y = \sum_i [y_i] \varpi^i$ in $\BW(R)[1/\varpi]$. 
\end{proof}



\begin{proposition}\label{prop:Lus2_6}
Suppose that $b \in T(\breve k)$ and $w$ is a $\sigma$-Coxeter element. Then 
\[ 
X_w(b)' \stackrel{\sim}{\longrightarrow} X_w(b), \quad u \mapsto u\dot w_0 B
\]
is an $B_b(k)$-equivariant isomorphism. Moreover, it induces a $T(k)$-equivariant isomorphism
\[ 
X_w(b)/U_b(k) \cong \prod_{i=1}^n LU_{-w_0(\alpha_i)}^\ast. 
\] 
\end{proposition}
\begin{proof}
By Proposition \
\ref{prop:Coxeter_DLspace_factors_big_cell_loop}, we have the inclusion $X_w(b) \hookrightarrow L(B \dot w_0 B/B) \stackrel{\sim}{\leftarrow} LU$, where the second isomorphism is $u \mapsto u\dot w_0$. This realizes $X_w(b)$ as a subsheaf of $LU$.
To show that it agrees with $X_w(b)'$, we compute for any $R \in {\rm Perf}_{\overline\BF_q}$ and any $u \in LU(R)$:
\begin{align*}
u\dot w_0 LB \in X_w(b)(R) &\Leftrightarrow \dot w_0^{-1} u^{-1} b F(u) F(\dot w_0) \in L(b^{w_0}\cdot U^- \times_{G} B \dot w B)(R) = \\ &\mbox{}\hspace{3.6cm}\stackrel{\text{Lm.~\ref{lm:Lus2_2_loop}}}{=} b^{w_0} \cdot \prod_{i=1}^n LU^\ast_{-\alpha_i}(R) \\
&\Leftrightarrow u^{-1}bF(u) \in \dot w_0 b^{w_0} \cdot \prod_{i=1}^n LU^\ast_{-\alpha_i}(R) \dot w_0^{-1} = b \cdot \prod_{i=1}^n LU^\ast_{-w_0(\alpha_i)}(R), 
\end{align*}
where we use that $L$ commutes with finite products.
The $B_b(k)$-equivariance is immediate, and the last claim is immediate from \eqref{eq:def_Xwbprime_loop}.
\end{proof}

We mention the following special case of Proposition \ref{prop:Lus2_6}, which is new to the $p$-adic setting due to the presence of regular $\sigma$-conjugacy classes. Recall that a $\sigma$-conjugacy class $[b] \subseteq G(\breve k)$ is called \emph{regular}, if there is some $\mu \in X_\ast(T)$ with $[b] = [\varpi^\mu]$ and $\langle \alpha,\mu\rangle \neq 0$. If $b$ is regular, then $G_b = B_b = T$ and $U_b = 1$ (the latter also follows from Lemma \ref{lm:LUrarbLU}). 

\begin{corollary}
Assume $b \in T(\breve k)$ is regular. Then the map in \eqref{eq:def_Xwbprime_loop} induces a $G_b(k)=T(k)$-equivariant isomorphism $X_w(b) \cong \prod_{i=1}^n LU_{-w_0(\alpha_i)}^\ast$. 

In particular, $X_w(b)$ is a disjoint union of affine schemes, there is a $T(k)$-equivariant isomorphism $\pi_0(X_w(b)) \cong X_\ast(T_{\rm ad})$ and any connected component of $X_w(b)$ is isomorphic to $L^+\BG_m$.
\end{corollary}
\begin{proof}
This follows from Proposition \ref{prop:Lus2_6} as $U_b=1$ and $G_b = B_b = T$.
\end{proof}

\subsection{$U_b(k)$-quotient of $\dot X_{\dot w}(b)$}\label{quotient_description_covers_loop}

Now we deduce an analogue of Proposition \ref{prop:Lus2_6} for the spaces $\dot X_{\dot w}(b)$. Let $\widetilde W = N_G(T)(\breve k)/T(\CO_{\breve k})$ be the extended affine Weyl group and let 
\[
F_w = \text{fiber of $N_G(T) \to W$ over $w$} \quad \text{ and } \quad \overline F_w = \text{fiber of $\widetilde W \to W$ over $w$}
\]
We regard $F_w$ as a trivial $T$-torsor over $\breve k$ and $\overline F_w$ as a trivial $X_\ast(T)$-torsor over $\overline\BF_q$ (in particular, $\pi_0(LF_w) = \overline F_w$). Recall from \cite[\S10]{Ivanov_DL_indrep} that we have maps $\kappa_w \colon LT \twoheadrightarrow X_\ast(T) \stackrel{\bar\kappa_w}{\twoheadrightarrow} X_\ast(T)_{\langle\sigma_w\rangle}$, where $\sigma = q^{-1}F$ is the automorphism of $X_\ast(T)$, and $\sigma_w = {\Ad}(w) \circ \sigma$ and that we have a natural map $\alpha_{w,b} \colon X_w(b) \to L\overline F_w/\ker\overline\kappa_w$ so that for any $\bar w \in L\overline F_w/\ker\overline\kappa_w$, $X_w(b)_{\bar w} = \alpha^{-1}(\bar w)$ is clopen (possibly empty) $G_b(k)$-stable subset of $X_w(b)$ and $\dot X_{\dot w}(b) \to X_w(b)_{\bar w}$ is a pro-\'etale $T_w(k)$-torsor for any lift $\dot w$ of $\bar w$.

For a root $\alpha \in \Phi$, let $s_\alpha \in W$ denote the corresponding reflection. 
As in \cite[6.1.2(2)]{BruhatT_72} we have the $\breve k$-subscheme $M_{\alpha}^\circ \subseteq F_{s_\alpha}$ and an isomorphism
\[
m=m_\alpha \colon U_{-\alpha}^\ast \stackrel{\sim}{\longrightarrow} M_\alpha^\circ \quad u \mapsto m(u),
\]
where for $u \in U_{-\alpha}(\breve k)$, $m(u)$ is the unique element of $F_w(\breve k)$ such that $u \in U_\alpha(\breve k) m(u) U_\alpha(\breve k)$.

\begin{lemma}\label{lm:intersection_Uminus_UwU}
Let $w = s_1\dots s_n$ be as in the beginning of \S\ref{sec:pDL_embedding_into_big_cell_loop}. Let $\dot w \in F_w(\breve k)$. For $\tau \in T(\breve k)$ we have 
\begin{equation}\label{eq:intersection_Uminus_UwU}
\tau U^- \cap U \dot w U = \tau \cdot \left\{ \prod_{i=1}^n v_i \in \prod_{i=1}^n U_{-\alpha_i}^\ast \colon \prod_{i=1}^n m(v_i) = \tau^{-1}\dot w \right\}.
\end{equation}
\end{lemma}
\begin{proof}
Multiplying both sides by $\tau^{-1}$ and using that $\tau^{-1}$ normalizes $U$, we may assume that $\tau = 1$. For better readability, we write $U$ for $U(\breve k)$ in the proof. As $v_i \in U m(v_i) U$ by construction of $m(\cdot)$, and as $U \dot y_1 U \dot y_2 U = U\dot y_1\dot y_2U$ whenever $y_1,y_2 \in W$ and $\dot y_i \in F_{y_i}$ with $\ell(y_1) + \ell(y_2) = \ell(y_1y_2)$, the right side of \eqref{eq:intersection_Uminus_UwU} is contained in the left side. For the converse, if $x \in U^- \cap U \dot w U$, then by Lemma~\ref{lm:Lus2_2_loop}, $x = \prod_{i=1}^n v_i$ with $v_i \in U^\ast_{-\alpha_i}$ and if $\prod_{i=1}^n m(v_i) = \dot w'$ for some $\dot w' \in F_w$, the above argument shows that $x \in U^- \cap U \dot w' U$, so we must have $\dot w' = \dot w$.
\end{proof}

\emph{From now on assume that $w = s_1\dots s_n$ is a $\sigma$-Coxeter element and that $b \in T(\breve k)$.} Consider the closed sub-ind-scheme of $LT \times \prod_{i=1}^n LU^\ast_{-w_0(\alpha_i)}$,
\[
Z_{\dot w}(b) = \left\{ (\tau, (v_i)_{i=1}^n) \in LT \times \prod_{i=1}^n LU^\ast_{-w_0(\alpha_i)} \colon \tau \dot w \sigma(\tau)^{-1} = (b^{w_0})\cdot \prod_{i=1}^n m(v_i) \right\}.
\]
The map 
\[ 
Z_{\dot w}(b) \to b \cdot \prod_{i=1}^n LU_{-w_0(\alpha_i)}^\ast \stackrel{\text{Prop.\ref{prop:Lus2_6}}}{\cong} X_w(b)/U_b(k), \quad (\tau,v_1,\dots,v_n) \mapsto b\cdot \prod_{i=1}^n v_i
\]
realizes $Z_{\dot w}(b)$ as a pro-\'etale $T_w(k)$-torsor over a clopen subset of the target.

\begin{proposition}\label{prop:description_dotXwb_quot}
Let $w$ be a $\sigma$-Coxeter element and $b \in T(\breve k)$. Then $\dot X_{\dot w}(b)/U_b(k) \cong Z_{\dot w}(b)$ and there is a cartesian diagram
\[
\xymatrix{
\dot X_{\dot w}(b) \ar[r] \ar[d]^{T_w(k)} & Z_{\dot w}(b) \ar[d]^{T_w(k)} \\
X_w(b) \ar[r] & X_w(b)/U_b(k)
}
\]
\end{proposition}
With other words, $\dot X_{\dot w}(b)$ is $B_b(k) \times T_w(k)$-equivariantly isomorphic to the set of all $(\tau,u) \in LT \times LU$ for which $u^{-1}b\sigma(u) = b \cdot \prod_{i=1}^n v_i \in \prod_{i=1}^n LU_{-w_0(\alpha_i)}^\ast$ and $\tau \dot w \sigma(\tau)^{-1} = b^{w_0}\cdot  \prod_{i=1}^n m(v_i)$.


\begin{proof}
As $\sigma(w_0)=w_0$, we may (using Lang's theorem) choose a lift $\dot w_0 \in N_{G}(T)(\breve k)$ with $\sigma(\dot w_0) = \dot w_0$. By Proposition \ref{prop:Coxeter_DLspace_factors_big_cell_loop}, $\dot X_{\dot w}(b) \hookrightarrow L(G/U)$ factors through the preimage $L(U\dot w_0B/U)\subseteq L(G/U)$ of $L(U\dot w_0 B/B)$. Now, $u\dot w_0 \tau LU \in L(U\dot w_0B/U) $ lies in $\dot X_{\dot w}(b)$ if and only if $(u\dot w_0 \tau)^{-1} b \sigma(u\dot w_0 \tau) \in L(U \dot w U)$, or equivalently, if and only if $(u^{-1}b\sigma(u))^{\dot w_0} \in L(U \tau \dot w \sigma(\tau)^{-1} U)$. As $(u^{-1}b\sigma(u))^{\dot w_0} \in b^{w_0} \cdot LU^-$, the last condition is equivalent to
\[
\left(u^{-1}b\sigma(u)\right)^{\dot w_0} \in L\left(b^{w_0}U^- \cap U \tau \dot w \sigma(\tau)^{-1} U\right) 
\]
which by Lemma \ref{lm:intersection_Uminus_UwU} is equivalent to
\[
\left(u^{-1}b\sigma(u)\right)^{\dot w_0} \in b^{w_0} \cdot \left\{\prod_i v_i \in \prod_{i=1}^n LU_{-\alpha_i}^\ast \colon \prod_{i=1}^R m(v_i) = (b^{w_0})^{-1} \tau \dot w \sigma(\tau)^{-1} \right\},
\]
Conjugating both sides by $\dot w_0$, and renaming the variable ${}^{\dot w_0} v_i \in LU_{-w_0(\alpha_i)}^\ast$ to $v_i$, we obtain the proposition.
\end{proof}

As a corollary we obtain a description of the map $\alpha_{w,b}$ from \cite[\S10]{Ivanov_DL_indrep} in this case. Denote by $\psi$ the map $\prod_{i=1}^n LU^\ast_{\alpha_i} \to F_w$, $(v_i)_{i=1}^n \mapsto \prod_{i=1}^n m(v_i)$.

\begin{lemma}\label{lm:surjectivity_sssc_case}
Let $w = s_1\dots s_n$ be $\sigma$-Coxeter. The image of the composed map
\[
\bar{\psi} \colon 
\prod_{i=1}^n LU_{-\alpha_i}^\ast \stackrel{L\psi}{\longrightarrow} LF_w \to L\overline F_w/\ker\overline\kappa_w, \quad  (v_i)_{i=1}^n \mapsto \prod_{i=1}^n m(v_i) 
\]
is equal to ${\rm im}(L\overline F_w^{\rm sc}/\ker\overline \kappa_w^{\rm sc} \to L\overline F_w/\ker\overline \kappa_w)$, where $(\cdot)^{\rm sc}$ denote the object $(\cdot)$ for the simply connected cover of the derived group of $G$.
\end{lemma}
\begin{proof}
We may assume that $G$ is semisimple and simply connected.
Then it suffices to show that $\prod_{i=1}^n LM_{\alpha_i}^\circ \to LF_w \to LF_w/\ker\kappa_w$ is surjective. 
For any coroot $\alpha^\vee \in \Phi^\vee$, let $T_{\alpha^\vee} \subseteq T$ denote its image. 
For each $i$ choose some $\dot s_i \in M_{\alpha_i}^\circ(\breve k)$, so that $M_{\alpha_i}^\circ = T_{\alpha_i^\vee}\dot s_i$. Then $\dot w := \dot s_1 \dot s_2 \dots \dot s_n \in F_w$. 
For $0\leq i<n$ let $\theta_i = s_1 \dots s_i(\alpha_{i+1})$ (we have $\{\theta_i\}_{i=0}^{n-1} = \Phi \cap {}^{w^{-1}}\Phi^-$, cf. \cite[p.~158]{Bourbaki_68}). Trivializing all torsors reduces us to showing that the natural map
\[
LT_{\theta_0^\vee} \times LT_{\theta_1^\vee} \times \dots \times LT_{ \theta_{n-1}^\vee}  \to LT/\ker\kappa_w \cong X_\ast(T)_{\langle \sigma_w \rangle}
\]
is surjective. It suffices to show that the composition
\[
X_\ast(T_{\theta_0^\vee}) \times X_\ast(T_{\theta_1^\vee}) \times \dots \times X_\ast(T_{ \theta_{n-1}^\vee})  \stackrel{\phi}{\to} X_\ast(T) \twoheadrightarrow X_\ast(T)_{\langle \sigma_w \rangle}
\]
is surjective As $G^{\rm sc} = G$, we have $X_\ast(T) = \BZ\Phi^\vee$ and the claim follows.
\end{proof}

Note that each $X_w(b)_{\bar w}$ is $G_b(k)$-stable, $X_w(b)/U_b(k)$ is the disjoint union of the clopen pieces $X_w(b)_{\bar w}/U_b(k)$.  Let $\bar\psi$ be as in Lemma \ref{lm:surjectivity_sssc_case}.

\begin{corollary}
Let $w$ be a $\sigma$-Coxeter element and $b \in T(\breve k)$.  Let $\bar w \in L\overline F_w/\ker \bar\kappa_w$. Under the isomorphism from Proposition \ref{prop:Lus2_6}, $X_w(b)_{\bar w}/U_b(k)$ corresponds to the subset of $\prod_{i=1}^n LU_{-w_0(\alpha_i)}^\ast$ cut out by the equation $\bar w = \bar\kappa_w(b) \bar\psi(\prod_i m(v_i)) \in L\overline F_w/\ker \bar\kappa_w$. In particular, \[ {\rm im}(\alpha_{w,b}) = \bar\kappa_w(b) \cdot {\rm im}(L\overline F_w^{\rm sc}/\ker\overline \kappa_w^{\rm sc} \to L\overline F_w/\ker\overline \kappa_w).\]
\end{corollary}
\begin{proof}
Passing to $LF_w / \ker\kappa_w$, the equation in the definition of $Z_{\dot w}(b)$ becomes $\bar w = \bar\kappa_w(b) \bar\psi(\prod_i m(v_i))$. By Proposition \ref{prop:description_dotXwb_quot} all claims follow from this and Lemma \ref{lm:surjectivity_sssc_case}.
\end{proof}

\subsection{Quotients by the unipotent radical of a parabolic}\label{sec:quot_mod_unip_rad_parabolic}

Let $I \subseteq S/\langle \sigma\rangle$ be a subset, let $S_I \subseteq S$ be its preimage in $S$, $W_I \subseteq W$ be the corresponding parabolic subgroup; $P_I$ the unique parabolic subgroup of $G$ containing $B$ and $U_{-\alpha}$ for all $\alpha \in I$; $U_I$ the unipotent radical of $P_I$; $G_I$ the unique Levi subgroup of $P_I$ containing $T$. Then $P_I, U_I, G_I$ are $k$-rational. Let $G_I' = P_I/U_I$ and denote the natural projection by $\pi \colon P_I \to G'_I$; the composition $G_I \to P_I \stackrel{\pi}{\to} G_I'$ is an isomorphism. Note that $\pi(B)$ is a Borel subgroup of $G_I'$. Let $\Pi_I \subseteq \Pi$ be the set of simple roots $\alpha$ corresponding to elements of $S_I$. Put $\Phi_I = \BZ \Pi_I \cap \Phi$ and $\Phi_I^{\pm} = \Phi_I \cap \Phi^{\pm}$. 

Write $n = |S/\langle\sigma\rangle|$ and assume that $|I| = n-1$. Let $w = s_1\dots s_n \in W$ be a $\sigma$-Coxeter element. Then there is a unique index $j$, such that ${}^{w_0}s_j \not\in S_I$. Let $w_0^I$ denotes the longest element of $W_I$. Then ${}^{w_0^I w_0}s_i \in S_I$ for all $i\neq j$. Thus \[w_I = {}^{w_0^I w_0}(s_1\dots s_{j-1} s_{j+1} \dots s_n)\] is a $\sigma$-Coxeter element of $W_I$. Denote by $X_{w_I}^{G'_I}(b)$ the corresponding $p$-adic Deligne--Lusztig space for the group $G'_I$.

\begin{lemma}\mbox{}\label{lm:describe_quotient_mod_parabolic_radical}
\begin{itemize}
\item[(i)] We have a well-defined map $G/B \supseteq B w_0 B/B \to G_I'/\pi(B)$ defined by sending $b_1 \dot w_0 B$ to $\pi(b_1) \dot w_0^I \pi(B)$. \item[(ii)] Let $b \in T(\breve k)$. The restriction of the map from (i) to $X_w(b)$ defines a $P_{I,b}(k)$-equivariant map $X_w(b) \to X_{w_I}^{G_I'}(b)$, where $P_{I,b}(k)$ acts on $X_{w_I}^{G_I'}(b)$ via its quotient in $G_I(k)$. This induces a $G'_b(k)$-equivariant map \[ \pi'' \colon X_w(b)/U_{I,b}(k) \to X_{w_I}^{G_I'}(b).\]
\item[(iii)] There exists a map $\alpha \colon X_w(b) \to LU_{-w_0(\alpha_j)}^\ast$, such that $\pi'' \times \alpha \colon X_w(b)/U_{I,b}(k) \to X_{w_I}^{G_I'}(b) \times LU_{-w_0(\alpha_j)}^\ast$ is an isomorphism.
\end{itemize}
\end{lemma}
\begin{proof}
(i) is an immediate computation. Then (ii) follows from Proposition \ref{prop:Lus2_6}. The proof of (iii) is the same as that of the first claim of \cite[Corollary 2.10]{Lusztig_76_Inv}.
\end{proof}

We explicate the isomorphism of Lemma \ref{lm:describe_quotient_mod_parabolic_radical}(iii). By Proposition \ref{prop:Lus2_6}, $X_w(b)$ identifies with the set of all $u \in LU$ satisfying $u^{-1}\dot w F(u) = b^{\dot w_0} \prod_{i=1}^n u_i$ with $u_i \in LU_{-w_0(\alpha_i)}^\ast$. There is a unique writing $u=u' u''$ with $u' \in U_I$ and $u'' \in U \cap G_I$. Then $u^{-1}bF(u) = u^{\prime\prime -1} bg F(u^{\prime\prime})$ where $u' \mapsto g = b^{-1}u^{\prime -1} bF(u') \colon LU \to LU$ defines a pro-\'etale $U_{I,b}(k)$-torsor. Thus $X_w(b)/U_{I,b}(k)$ identifies with the set of all $(u'',g)$ such that \begin{equation}\label{eq:quot_parabolic_aux1} 
u^{\prime\prime-1}bgF(u^{\prime\prime}) = \prod_{i=1}^n u_i \quad \text{ with $u_i \in LU_{-w_0(\alpha_i)}^\ast$.}
\end{equation}
Applying $\pi$ and noting that it induces an isomorphism $G_I \to G_I'$, we see that all $u_i$ for $i\neq j$ are uniquely determined by $u''$ and that $u'' \in X_{w_I}^{G_I'}(b)$ (under the identification of Proposition \ref{prop:Lus2_6}). It follows then that those $g \in U_I$ for which \eqref{eq:quot_parabolic_aux1} holds, are in bijection with all $u_j \in LU_{-w_0(\alpha_j)}^\ast$, so that $(u'',g) \mapsto (u'',u_j)$ is an isomorphism as claimed in Lemma \ref{lm:describe_quotient_mod_parabolic_radical}(iii).

Now we describe the quotient $\dot X_{\dot w}(b)/U_{I,b}(k)$. Consider 
\begin{align*}
Z_{I,\dot w}(b) = \{ (\tau,&u'',u_j) \in LT \times L(U \cap G_I) \times LU_{-w_0(\alpha_j)} \colon \\ &\tau^{-1} \dot w F(\tau) = b^{\dot w_0} \prod_{i=1}^n m(u_i), \,\, u^{\prime\prime -1}b F(u^{\prime\prime}) = \prod_{\substack{i=1 \\ i\neq j}}^{n} u_i\in \prod_{i\neq j} LU_i^\ast \},
\end{align*}
where $u_i$ ($i\neq j$) are determined by $u''$ as above, and the last equality takes place in $G_I$. Then $Z_{I,\dot w}(b)$ is a pro-\'etale $T_w(k)$-torsor over $X_{w_I}^{G_I'}(b) \times LU_{-w_0(\alpha_j)}^\ast$.

\begin{lemma}\label{lm:quot_by_UI_of_covers}
There is an $T(k)$-equivariant isomorphism $\dot X_{\dot w}(b)/U_{I,b}(k) \cong Z_{I,\dot w}(b)$, and $Z_{I, \dot w}(b)$ fits into the diagram with cartesian squares,
\[
\xymatrix{
\dot X_{\dot w}(b) \ar[r] \ar[d] &Z_{I,\dot w}(b) \ar[d] \ar[r] & Z_{\dot w}(b) \ar[d] \\
X_w(b) \ar[r] &X_w(b)/U_{I,b}(k) \ar[r] & X_w(b)/U_b(k)
}
\]
where the outer square is as in Proposition \ref{prop:description_dotXwb_quot}, and all vertical maps are pro-\'etale $T_w(k)$-torsors.
\end{lemma}
\begin{proof} 
It is clear from the definitions, that the right square is cartesian. This implies that there is a natural map $\dot X_{\dot w}(b) \to Z_{I,\dot w}(b)$. As the outer square is cartesian by Proposition \ref{prop:description_dotXwb_quot}, the left square has to be cartesian too. This implies the first claim of the lemma, and the $T_w(k)$-equivariance is clear.
\end{proof}

\section{Quotients on the integral/finite level}\label{sec:quots_finite_level}

We investigate the analogues of the results from \S\ref{sec:quotient_Cox_var} for deep level Deligne--Lusztig varieties. We assume that $b = 1$ (only possibility with $b \in T(\breve k)$ and basic). Let $\CG$ be a hyperspecial model of $G$ over $\CO_k$. Let $w = s_1 \dots s_n \in W$ be $\sigma$-Coxeter element and let $\dot w \in \CG(\CO_{\breve k})$ be a lift of $w$. Let $\BG = \BG_r (= L_r^+\CG)$ with $r\leq \infty$ be as in the introduction. We have the Deligne--Lusztig variety $X_{w} = X_{w,r} \subseteq \BG/\BB$ and the $\BT_r^F$-torsor $\dot X_{\dot w} = \dot X_{\dot w,r} \subseteq \BG/\BU$ over it (as in \cite[Definition 4.1.1]{DI}). Note that by \cite[Lemma 4.1.2]{DI} there is an $\BG^F \times \BT^F$-equivariant isomorphism 
\begin{equation}\label{eq:two_setups} 
\dot X_{\dot w,r} \cong X = \{ g\in \BG \colon g^{-1}F(g) \in \overline\BU \cap F\BU \}. 
\end{equation}
Let $\pi \colon \BG \to \BG_1$ denote the natural projection map. 

\begin{lemma}\label{lm:contained_in_big_cell}
We have $X_{w,r} \subseteq \BB \dot w_0 \BB$.
\end{lemma}

\begin{proof}
$\pi(X_{w,r}) \subseteq X_{w,1}$ and $X_{w,1} \subseteq \BB_1 \dot w_0 \BB_1$ by \cite[Cor.~2.5]{Lusztig_76_Inv}. Thus $X_{w,r} \subseteq \pi^{-1}(\BB_1 \dot w_0 \BB_1) = \BB \dot w_0 \BB$, the last equality being true since $w_0$ is the longest element of $W$.
\end{proof}

\begin{lemma}\label{lm:Bruhat_properties} Let $v,v' \in W$ with $\ell(vv') = \ell(v)+\ell(v')$. Then $\BB \dot v \BB \dot v' \BB = \BB \dot v \dot v' \BB$.
\end{lemma}
\begin{proof}
Let $\alpha \in \Phi^+$ be the simple root corresponding to $s$. For part (i), we are reduced by induction to the case that $v' = s \in S$ and $\ell(vs) = \ell(v) + 1$. Then we have $v(\alpha) \in \Phi^+$.
It follows that 
$
\BB \dot v \BB \dot s \BB = \BB \dot v \BU_{\alpha} \dot s \BB = \BB \dot v \dot s \BB,
$
where in the first step we move all $\BU_{\beta}$ with $\beta \neq \alpha$ into the right $\BB$, and in the second step we move $\BU_\alpha$ into the left $\BB$, using $v(\alpha) \in \Phi^+$.
\end{proof}

For $\alpha \in \Phi$, we write $\BU_\alpha^\ast$ for the open complement of $\BU^1_\alpha$ in $\BU_\alpha$. We have the following analogue of \cite[Proposition 2.2]{Lusztig_76_Fin}.

\begin{lemma}\label{lm:Uminus_intesect_CoxCell}
We have
\[
\BU^- \cap \BB \dot w \BB = \{ v_1 \dots v_n \colon v_i \in (\BU_{-\alpha_i})^\ast \},
\]
\end{lemma}
\begin{proof}
Let $v_i \in \BU_{-\alpha_i}^{\ast}$. Then we claim that $v_i \in \BB s_i \BB$. Indeed, let $G_{\alpha_i} \subseteq G$ be the subgroup generated by $U_{\alpha_i}, U_{-\alpha_i}, T$, and let $\BG_{\alpha_i} \subseteq \BG$ be the corresponding subgroup. Then $v_i \in \BG_{\alpha_i}$, and it suffices to show the claim for $\BG_{\alpha_i}$ instead of $\BG$, which reduces to an explicit computation in $L^+_r\SL_2$, which uses the assumption that $v_i \not\in \BU_{\alpha_i}^1$. Using the claim, 
\[
v = \prod_i v_i \in \BB \dot s_1 \BB \dot s_2 \dots \dot s_{n-1} \BB s_n \BB = \BB \dot w \BB,
\]
by Lemma \ref{lm:Bruhat_properties}. This shows one inclusion. For the converse, assume that $x \in \BU^- \,\cap \, \BB \dot w \BB$. Then just as in \cite[Proof of (2.2)]{Lusztig_76_Inv}, we may write $x = u_1 \dot s_1 u_2 \dot s_2 \dots u_n \dot s_n b$ with $u_i \in \BU_{\alpha_i}$ and $b \in \BB$. Suppose that $u_1 \in \BU_{\alpha_1}^1$. Then consider the image $\bar x \in \BU_{\alpha_1}^- \,\cap \, \BB_1 \dot w \BB_1$ of $x$ under $\pi \colon \BG \to \BG_1$, which is again of the form $\bar x = \bar u_1 \dot s_1 \bar u_2 \dot s_2 \dots \bar u_n \dot s_n \bar b$, with $\bar u_1 \not= 1$ in $(\BU_{\alpha_1})_1$. As in \emph{loc.~cit.} this gives a contradiction. Thus we must have $u_1 \in \BU_{\alpha_1}^\ast$. Similar as in \emph{loc.~cit.}, a computation in the group $\BG_{\alpha_i}$ shows that there exist $u_1' \in \BU_{\alpha_1}$, $v_1 \in (\BU_{\alpha_1}^-)^\ast$, $t \in \BT$ with $u_1 = v_1 u_1' t \dot s_1$. Using this, we see that
\[
v_1^{-1} x = u_1' t \dot s_1^2 u_2 \dot s_2 \dots u_n \dot s_n b  = u_1' t u_2 \dot s_2 \dots u_n \dot s_n b \in \BB \dot s_2 \BB \dots \dot s_n \BB = \BB \dot s_2 \dots \dot s_n \BB.
\]
Thus $v_1^{-1} x \in \BU^- \cap \BB \dot s_2 \dots \dot s_n \BB$, and we are done by induction. \qedhere
\end{proof}

Note that Lemma \ref{lm:Uminus_intesect_CoxCell} does not follow directly from \cite[Prop.~2.2]{Lusztig_76_Inv} as both sides of the equation are not equal to the preimages of their images under $\pi \colon \BG \to \BG_1$. Now we can generalize \cite[Theorem 2.6]{Lusztig_76_Inv}.

\begin{proposition}\label{prop:descr_of_Xwr_and_unip_quot} We have the following isomorphisms:
\begin{itemize}
\item[(i)] 
$\{  u \in \BU \colon u^{-1}F(u) = u_1 \dots u_n \colon u_i \in \BU_{-w_0(\alpha_i)}^\ast\, \forall\, 1\leq i\leq n \} \stackrel{\sim}{\to} X_{w,r}, \quad u \mapsto u\dot w_0 \BB$
\item[(ii)] $\BU_{-w_0(\alpha_1)}^\ast \times \dots \times \BU_{-w_0(\alpha_n)}^\ast \stackrel{\sim}{\to} X_{w,r}/\BU^F$.
\end{itemize}
\end{proposition}
\begin{proof}
This follows from Lemmas \ref{lm:contained_in_big_cell} and \ref{lm:Uminus_intesect_CoxCell}.
\end{proof}


Let $\alpha \in \Phi$. The map $m = m_\alpha \colon U^\ast_{-\alpha} \to M_\alpha^\circ$ from \S\ref{quotient_description_covers_loop} induces an isomorphism
\[
m_\alpha \colon \BU_{-\alpha}^\ast \stackrel{\sim}{\to}\BM_\alpha^\circ,
\]
where $\BM_{\alpha}^\circ$ is the preimage of $w$ in $\BG$. Then just as in \S\ref{quotient_description_covers_loop} we have the scheme with $\BB^F \times \BT_w^F$-action
\[Z_{\dot w,r} = \{(\tau,(v_i)_{i=1}^n) \in \BT \times \prod_{i=1}^n \BU^\ast_{-w_0(\alpha_i)} \colon \tau \dot w \sigma(\tau)^{-1} = \prod_{i=1}^n m(v_i) \} \]
equipped with the $\BB^F$-equivariant map $(\tau,(v_i)_i) \mapsto (v_i)_i \colon Z_{\dot w,r} \to X_{w,r}/\BU^F$. With notation as in \S\ref{sec:quot_mod_unip_rad_parabolic} we also have the scheme 
\begin{align*}
Z_{I,\dot w,r}(b) = \{ (\tau,&u'',u_j) \in \BT \times (\BU \cap \BG_I) \times \BU_{-w_0(\alpha_j)} \colon \\ &\tau^{-1} \dot w F(\tau) = b^{\dot w_0} \prod_{i=1}^n m(u_i), \,\, u^{\prime\prime -1}b F(u^{\prime\prime}) = \prod_{\substack{i=1 \\ i\neq j}}^{n} u_i\in \prod_{i\neq j} \BU_i^\ast \},
\end{align*}
and just as in \S\ref{sec:quot_mod_unip_rad_parabolic} we have the following consequence of Proposition \ref{prop:descr_of_Xwr_and_unip_quot}.

\begin{corollary}\mbox{}
\begin{itemize}
\item[(1)] There is $\BB^F$-equivariant isomorphism
\[
\dot X_{\dot w,r}/\BU^F \stackrel{\sim}{\to} Z_{\dot w,r}
\]
and $X_{\dot w,r} = X_{w,r} \times_{X_{w,r}/\BU^F}  Z_{\dot w,r}$.
\item[(2)] There is $\BT_w^F$-equivariant isomorphism
\[
\dot X_{\dot w,r}/\BU_I^F \stackrel{\sim}{\to} Z_{I,\dot w,r}
\]
and $\dot X_{\dot w,r} = X_{w,r} \times_{X_{w,r}/\BU_I^F}  Z_{I,\dot w,r}$.
\end{itemize}
\end{corollary}

\subsection{Extension of action}
Let the notation be as in \S\ref{sec:quot_mod_unip_rad_parabolic}. Let $\mu \in X_\ast(T)$ be be such that $\langle \alpha_j, \mu\rangle \neq 0$ (such $\mu$ exists). Put $\mu' = \langle \alpha_j,\mu\rangle \cdot {}^{s_1\dots s_{j-1}}(\alpha_j^\vee) \in X_\ast(T)$. As $w\sigma - 1 \colon X_\ast(T^{\rm sc})_\BQ \to X_\ast(T^{\rm sc})_\BQ$ is bijective, we may (after replacing $\mu$ by an integral multiple, if necessary) assume that there is some $\lambda \in X_\ast(T^{\rm sc})$ with $\mu' = w\sigma(\lambda) - \lambda$. 

\begin{lemma}\label{lm:action_of_Gm}
With notation as above, there is an action of $\BG_m$ on $Z_{I,\dot w,r}$ given by the formula
\[
x \colon (\tau,u'',u_j) \mapsto (\lambda(x)\tau,u'',{}^{\mu(x)}u_j)
\]
for any $x \in \BG_m$. Moreover, this $\BG_m$-action commutes with the action of $\BT_w^F$ and $Z_{I,\dot w,r}^{\BG_m} = \varnothing$.
\end{lemma}

\begin{proof}
Let $(\tau,u'',u_j) \in Z_{I,\dot w,r}$ and let $u_i \in \BU_{-w_0(\alpha_i)}$ (for $i\neq j$) be determined by $u''$ as above. The first sentence of the lemma follows from the computation
\begin{align*}
(\lambda(x)\tau)^{-1} \dot w F(\lambda(x)\tau) &= \mu'(x) \tau^{-1}\dot w F(\tau) \\ &= \mu'(x) \prod_{i=1}^n m(u_i) \\
&= \prod_{i=1}^{j-1} m(u_i) (\alpha_j^\vee)^{\langle\alpha_j,\mu\rangle}(x)m(u_j) \prod_{i=j+1}^n m(u_i) \\
&= \prod_{i=1}^{j-1} m(u_i) m({}^{\mu(x)}u_j) \prod_{i=j+1}^n m(u_i)
\end{align*}
where the last equality follows from a property of the map $m(\cdot)$, which can be checked by an explicit calculation after reducing to $\SL_2$. The last sentence of the lemma is immediate.
\end{proof}

\begin{proof}[Proof of Theorem \ref{thm:cusp}]
By \cite[Corollary 1.0.1]{DI} (or Theorem \ref{thm:main}) and \eqref{eq:two_setups} we know that $H_c^\ast(X)[\chi] = H_c^\ast(\dot X_{\dot w,r})[\chi]$ is up to sign an irreducible $\BG^F$-representation. Thus, exploiting a theorem of Bushnell \cite[Theorem 1]{Bushnell_90} as in the proof of \cite[Theorem 6.1, Proposition 6.2]{CI_loopGLn}, it suffices to show that for any maximal proper subset $I \subseteq S/\langle \sigma \rangle$, the virtual $\overline\BQ_\ell$-vector space $H_c^\ast(X_{\dot w,r}^\CG/\BU_I^F,\overline\BQ_\ell)_\theta$ vanishes. But this follows directly from Lemma \ref{lm:action_of_Gm}, as 
\[ 
H_c^\ast(X_{\dot w,r}^\CG/\BU_I^F,\overline\BQ_\ell)_\theta = H_c^\ast(Z_{I,\dot w,r},\overline\BQ_\ell)_\theta = H_c^\ast(Z_{I,\dot w,r}^{\BG_m},\overline\BQ_\ell)_\theta = 0,
\] 
where the last equality follows from \cite[10.15]{DigneM_91}.
\end{proof}

\bibliography{bib_ADLV}{}
\bibliographystyle{amsalpha}

\end{document}